\documentclass[12pt, a4paper]{amsart}

\usepackage{amssymb}

\usepackage{graphicx}
\usepackage{amsthm,amssymb, amstext, amscd,  amsmath,  color}

\usepackage{tikz,enumitem,tikz-3dplot,tikz-cd}
\usepackage{pgfplots}

\definecolor{cof}{RGB}{219,144,71}
\definecolor{pur}{RGB}{186,146,162}
\definecolor{greeo}{RGB}{91,173,69}
\definecolor{greet}{RGB}{52,111,72}

\tdplotsetmaincoords{70}{165}

\usetikzlibrary{arrows,patterns}
\usepackage{epstopdf}
\usepackage{url}
\usepackage{kbordermatrix} 
\renewcommand{\kbldelim}{[} 
\renewcommand{\kbrdelim}{]}
\usepackage{hhline}

\begin{document}

\newtheorem{theorem}{Theorem}
\newtheorem{lemma}[theorem]{Lemma}
\newtheorem{corollary}[theorem]{Corollary}
\newtheorem{proposition}[theorem]{Proposition}
\newtheorem{fact}[theorem]{Fact}
\newtheorem{observation}[theorem]{Observation}
\newtheorem{claim}[theorem]{Claim}

\newtheorem{definition}[theorem]{Definition}
\newtheorem{example}[theorem]{Example}
\newtheorem{conjecture}[theorem]{Conjecture}
\newtheorem{open}[theorem]{Open Problem}
\newtheorem{problem}[theorem]{Problem}
\newtheorem{question}[theorem]{Question}

\newtheorem{remark}[theorem]{Remark}
\newtheorem{note}[theorem]{Note}

%

\newcommand{\FFock}{\mathcal{F}}
\newcommand{\kil}{\mathsf{k}}
\newcommand{\Hil}{\mathsf{H}}
\newcommand{\hil}{\mathsf{h}}
\newcommand{\Kil}{\mathsf{K}}
\newcommand{\Real}{\mathbb{R}}
\newcommand{\Rplus}{\Real_+}

\newcommand{\bC}{{\mathbb{C}}}
\newcommand{\bD}{{\mathbb{D}}}
\newcommand{\bN}{{\mathbb{N}}}
\newcommand{\bQ}{{\mathbb{Q}}}
\newcommand{\bR}{{\mathbb{R}}}
\newcommand{\bT}{{\mathbb{T}}}
\newcommand{\bX}{{\mathbb{X}}}
\newcommand{\bZ}{{\mathbb{Z}}}
\newcommand{\bH}{{\mathbb{H}}}
\newcommand{\BH}{{\B(\H)}}
\newcommand{\bsl}{\setminus}
\newcommand{\ca}{\mathrm{C}^*}
\newcommand{\cstar}{\mathrm{C}^*}
\newcommand{\cenv}{\mathrm{C}^*_{\text{env}}}
\newcommand{\rip}{\rangle}
\newcommand{\ol}{\overline}
\newcommand{\td}{\widetilde}
\newcommand{\wh}{\widehat}
\newcommand{\sot}{\textsc{sot}}
\newcommand{\wot}{\textsc{wot}}
\newcommand{\wotclos}[1]{\ol{#1}^{\textsc{wot}}}
 \newcommand{\A}{{\mathcal{A}}}
 \newcommand{\B}{{\mathcal{B}}}
 \newcommand{\C}{{\mathcal{C}}}
 \newcommand{\D}{{\mathcal{D}}}
 \newcommand{\E}{{\mathcal{E}}}
 \newcommand{\F}{{\mathcal{F}}}
 \newcommand{\G}{{\mathcal{G}}}
 \renewcommand{\H}{{\mathcal{H}}}
 \newcommand{\I}{{\mathcal{I}}}
 \newcommand{\J}{{\mathcal{J}}}
 \newcommand{\K}{{\mathcal{K}}}
 \renewcommand{\L}{{\mathcal{L}}}
 \newcommand{\M}{{\mathcal{M}}}
 \newcommand{\N}{{\mathcal{N}}}
 \renewcommand{\O}{{\mathcal{O}}}
 \renewcommand{\P}{{\mathcal{P}}}
 \newcommand{\Q}{{\mathcal{Q}}}
 \newcommand{\R}{{\mathcal{R}}}
 \renewcommand{\S}{{\mathcal{S}}}
 \newcommand{\T}{{\mathcal{T}}}
 \newcommand{\U}{{\mathcal{U}}}
 \newcommand{\V}{{\mathcal{V}}}
 \newcommand{\W}{{\mathcal{W}}}
 \newcommand{\X}{{\mathcal{X}}}
 \newcommand{\Y}{{\mathcal{Y}}}
 \newcommand{\Z}{{\mathcal{Z}}}

\newcommand{\supp}{\operatorname{supp}}
\newcommand{\conv}{\operatorname{conv}}
\newcommand{\cone}{\operatorname{cone}}
\newcommand{\vspan}{\operatorname{span}}
\newcommand{\proj}{\operatorname{proj}}
\newcommand{\sgn}{\operatorname{sgn}}
\newcommand{\rank}{\operatorname{rank}}
\newcommand{\Isom}{\operatorname{Isom}}
\newcommand{\qIsom}{\operatorname{q-Isom}}
\newcommand{\Cknet}{{\mathcal{C}_{\text{knet}}}}
\newcommand{\Ckag}{{\mathcal{C}_{\text{kag}}}}
\newcommand{\rind}{\operatorname{r-ind}}
\newcommand{\lind}{\operatorname{r-ind}}
\newcommand{\ind}{\operatorname{ind}}
\newcommand{\coker}{\operatorname{coker}}
\newcommand{\Aut}{\operatorname{Aut}}
\newcommand{\Hom}{\operatorname{Hom}}
\newcommand{\GL}{\operatorname{GL}}
\newcommand{\tr}{\operatorname{tr}}

\newcommand{\eqnwithbr}[2]{%
\refstepcounter{equation}
\begin{trivlist}
\item[]#1 \hfill $\displaystyle #2$ \hfill (\theequation)
\end{trivlist}}

\newcommand{\te}{\tilde{e}}
\newcommand{\tf}{\tilde{f}}
\newcommand{\tu}{\tilde{u}}
\newcommand{\tv}{\tilde{v}}
\newcommand{\ta}{\tilde{a}}
\newcommand{\tb}{\tilde{b}}
\newcommand{\tc}{\tilde{c}}
\newcommand{\tx}{\tilde{x}}
\newcommand{\ty}{\tilde{y}}
\newcommand{\tz}{\tilde{z}}
\newcommand{\tw}{\tilde{w}}

\title[Motions of grid-like reflection frameworks]{Motions of grid-like reflection frameworks}

\author[D. Kitson and  B. Schulze]{Derek Kitson  and  Bernd Schulze}
 \thanks{The first named author is supported by EPSRC grant  EP/P01108X/1.}
\thanks{The second named author is supported by EPSRC grant  EP/M013642/1.}
\email{d.kitson@lancaster.ac.uk, b.schulze@lancaster.ac.uk}
\address{Dept.\ Math.\ Stats.\\ Lancaster University\\
Lancaster LA1 4YF \\U.K. }


\begin{abstract} Combinatorial characterisations are obtained of symmetric and anti-symmetric  infinitesimal rigidity for two-dimensional frameworks with reflectional symmetry in the case of norms where the unit ball is a quadrilateral and where the reflection acts freely on the vertex set. At the framework level, these characterisations are given in terms of induced monochrome subgraph decompositions, and at the graph level they are given in terms of sparsity counts and recursive construction sequences for the corresponding signed quotient graphs.
\end{abstract}

\maketitle



\section{Introduction}
The objects considered in this article are geometric constraint systems where the constraints are determined by a possibly non-Euclidean choice of norm.
The main results are new contributions in both geometric and combinatorial rigidity. At the geometric level, characterisations are provided for rigid two-dimensional symmetric frameworks constrained by norms with a quadrilateral unit ball (the $\ell^1$ and $\ell^\infty$ norms for example). 
At the combinatorial level, the problem of deciding whether a graph can be realized as a forced symmetric or anti-symmetric isostatic reflection framework is considered and complete characterisations are obtained. 
Overall this article builds on recent work analyzing the rigidity of frameworks in normed linear spaces, with and without symmetry (see for example \cite{kitson,kit-pow,kit-sch,kit-sch1}). 

A bar-joint framework in the plane is referred to as {\em grid-like} if the bar-lengths are determined by a norm with a quadrilateral unit ball. 
The allowable motions of such a framework constrain vertices adjacent to any pinned vertex to move along the boundary of a quadrilateral which is centred at the pinned vertex and obtained from the unit ball by translation and dilation (see Fig.~\ref{fig:gridfw}).
This is an important context from the point of view of applications. 
For example, the problem of maintaining rigid formations of mobile autonomous agents is a well-known application of geometric rigidity theory and its associated ``pebble game" algorithms (see \cite{EAMWB}). 
However, the Euclidean metric may not always be the most natural choice for controlling a formation.
For instance, it may not be possible to detect Euclidean distances between agents  (eg. due to obstacles in the terrain). Moreover, if the agents have restricted mobility (eg. with only vertical and horizontal  directions of motion possible) then standard methods from Euclidean rigidity theory will have limited use. In these cases it may be desirable to have a rigidity theory for a non-Euclidean norm (such as the $\ell^1$ or $\ell^\infty$ norm) as an alternative approach to  formation control. 
An accompanying theory for {\em symmetric} frameworks may provide more efficient architectures for the control of formations due to the smaller size of the quotient graphs and their associated constraint systems.

 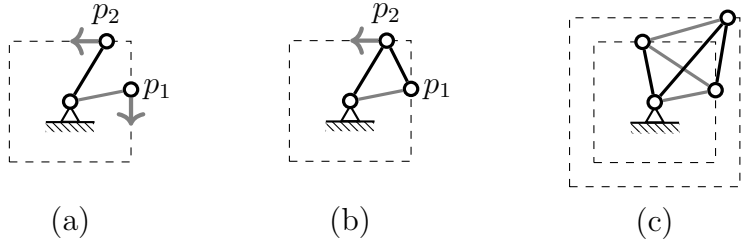
\begin{figure}[htp]
\begin{center}
\begin{tikzpicture}[very thick,scale=0.8]
\tikzstyle{every node}=[circle, draw=black, fill=white, inner sep=0pt, minimum width=5pt];

\draw[thin,dashed] (-1,-1)--(1,-1)--(1,1)--(-1,1)--(-1,-1);

\node [draw=white, fill=white] (a) at (1.44,0.2)  {$p_1$};
\node [draw=white, fill=white] (a) at (0.6,1.45)  {$p_2$};

\fill[draw=black, fill=white, thick, yshift=-0.25cm] (90:0.17cm) -- (210:0.17cm) -- (-30:0.17cm) -- cycle;
\draw[thin, draw=white, pattern=north west lines] (-0.4,-0.34)--(0.4,-0.34)--(0.4,-0.51)--(-0.4,-0.51)-- (-0.4,-0.34) ;
\draw[thick] (-0.4,-0.34)--(0.4,-0.34);
\node (p0) at (0,0) {};

\node (p1) at (1,0.2) {};
\node (p2) at (0.6,1) {};

\draw[gray](p0)--(p1);
\draw(p0)--(p2);

\draw[->,gray,ultra thick] (p1)--(1,-0.4);
\draw[->,gray,ultra thick] (p2)--(0,1);

\node [draw=white, fill=white] (a) at (0,-2)  {(a)};

\end{tikzpicture}
\hspace{1.2cm}
\begin{tikzpicture}[very thick,scale=0.8]
\tikzstyle{every node}=[circle, draw=black, fill=white, inner sep=0pt, minimum width=5pt];

\draw[thin,dashed] (-1,-1)--(1,-1)--(1,1)--(-1,1)--(-1,-1);

\node [draw=white, fill=white] (a) at (1.44,0.2)  {$p_1$};
\node [draw=white, fill=white] (a) at (0.6,1.45)  {$p_2$};

\fill[draw=black, fill=white, thick, yshift=-0.25cm] (90:0.17cm) -- (210:0.17cm) -- (-30:0.17cm) -- cycle;
\draw[thin, draw=white, pattern=north west lines] (-0.4,-0.34)--(0.4,-0.34)--(0.4,-0.51)--(-0.4,-0.51)-- (-0.4,-0.34) ;
\draw[thick] (-0.4,-0.34)--(0.4,-0.34);
\node (p0) at (0,0) {};

\node (p1) at (1,0.2) {};
\node (p2) at (0.6,1) {};

\draw[gray](p0)--(p1);
\draw(p0)--(p2);
\draw(p1)--(p2);

\draw[->,gray,ultra thick] (p2)--(0,1);

\node [draw=white, fill=white] (a) at (0,-2)  {(b)};

\end{tikzpicture}
\hspace{1.2cm}
\begin{tikzpicture}[very thick,scale=0.8]
\tikzstyle{every node}=[circle, draw=black, fill=white, inner sep=0pt, minimum width=5pt];

\draw[thin,dashed] (-1,-1)--(1,-1)--(1,1)--(-1,1)--(-1,-1);
\draw[thin,dashed] (-1.4,-1.4)--(1.4,-1.4)--(1.4,1.4)--(-1.4,1.4)--(-1.4,-1.4);

\fill[draw=black, fill=white, thick, yshift=-0.25cm] (90:0.17cm) -- (210:0.17cm) -- (-30:0.17cm) -- cycle;
\draw[thin, draw=white, pattern=north west lines] (-0.4,-0.34)--(0.4,-0.34)--(0.4,-0.51)--(-0.4,-0.51)-- (-0.4,-0.34) ;
\draw[thick] (-0.4,-0.34)--(0.4,-0.34);
\node (p0) at (0,0) {};

\node (p1) at (1,0.2) {};
\node (p2) at (-0.2,1) {};
\node (p3) at (1.2,1.4) {};

\draw[gray](p0)--(p1);
\draw(p0)--(p2);
\draw[gray](p1)--(p2);
\draw(p3)--(p0);
\draw(p3)--(p1);
\draw[gray](p3)--(p2);

\node [draw=white, fill=white] (a) at (0,-2)  {(c)};

\end{tikzpicture}
\end{center}
\vspace{-0.2cm}
\caption{Grid-like frameworks in $(\mathbb{R}^2,\|\cdot\|_\infty)$, where one of the vertices is fixed at the origin: the framework in (a) has two degrees of freedom, as $p_1$  and $p_2$ can move vertically and horizontally, respectively, independent of each other; the framework in (b) has one degree of freedom, as $p_2$ can still move horizontally; the framework in (c) is rigid. The colours of the edges are induced by their orientation relative to the unit ball in $(\mathbb{R}^2,\|\cdot\|_\infty)$.}
\label{fig:gridfw}
\end{figure}

There are three main aims of this article. The first is to formally introduce and develop symmetric and anti-symmetric infinitesimal rigidity for $\bZ_2$-symmetric frameworks in general normed linear spaces. This is achieved in Section \ref{Sect:orbit}.
Each infinitesimal flex is shown to decompose in a unique way as a sum of a symmetric and an anti-symmetric flex. Moreover, the rigidity operator is shown to admit a block decomposition which leads in a natural way to a consideration of orbit matrices. Sparsity counts, expressed in terms of an associated signed quotient graph, are then derived for symmetrically and anti-symmetrically isostatic frameworks. When applied to Euclidean frameworks, the block decomposition reduces to that studied in \cite{KG2,owen,BS2}, while the orbit matrices and sparsity counts coincide with those in \cite{jkt,schtan,BSWW}.

The second aim is to characterise symmetric, anti-symmetric and general infinitesimal rigidity for grid-like frameworks with reflectional symmetry, where the reflection acts {\em freely} on the vertex set. In Section \ref{Sect:framework}, characterisations are obtained in terms of edge colourings for the signed quotient graph. These edge colourings are induced from a symmetric edge-colouring of the covering graph which is in turn induced by the positioning of the framework relative to the unit ball. This may be viewed as an extension to symmetric frameworks of methods used in \cite{kitson,kit-pow}.   

The third aim, which is in the spirit of Laman's theorem (see \cite{Lamanbib,tay,W1}), is to provide combinatorial characterisations for graphs which admit placements as rigid grid-like frameworks with reflectional symmetry. 
This is achieved in Section \ref{Sect:graph} for both symmetric and anti-symmetric infinitesimal rigidity. The characterisations provide the sufficiency direction for the necessary sparsity counts derived in the general theory of Section \ref{Sect:orbit}. The proof applies an inductive construction for signed quotient graphs together with the results of Section \ref{Sect:framework}.
Note that these matroidal counts can be checked in polynomial time using a straightforward adaptation of the algorithm described in \cite[Sect.~10]{jkt} (see also \cite{bhmt}).

The results of Section \ref{Sect:graph} are analogous to the corresponding results for Euclidean reflection frameworks in \cite{jkt,MT2}. It is important to note, however, that unlike the Euclidean situation (see \cite{schtan}), the respective characterisations of graphs which admit symmetric or anti-symmetric  rigid placements as grid-like reflection frameworks cannot be combined to characterise graphs which admit rigid placements as grid-like reflection frameworks. This is due to the fact that the respective sets of symmetric and anti-symmetric rigid grid-like realisations of a graph may be disjoint (see Fig.~\ref{fig:antisym} for example). A combinatorial characterisation of graphs which admit a  realisation as a grid-like  \emph{isostatic}  reflection framework was recently given in \cite{kit-sch1}.  However, as shown in \cite{kit-sch,kit-sch1}, such  a framework must have a vertex which is fixed by the reflection.

\section{$\bZ_2$-symmetric frameworks in normed spaces}
\label{Sect:orbit}

Throughout this article $G=(V,E)$ will denote a finite simple undirected graph with vertex set $V$ and edge set $E$. 
An edge $e\in E$ which is incident to vertices $v,w\in V$ will be denoted $vw$. An automorphism of $G$ is a bijective map $h:V\to V$ with the property that $vw\in E$ if and only if $h(v)h(w)\in E$. The group (under composition) of graph automorphisms of $G$ is denoted $\Aut(G)$.
Consider the multiplicative group $\bZ_2$ with elements $\{1,-1\}$.
A {\em $\bZ_2$-symmetric graph} is a pair $(G,\theta)$ consisting of a graph $G$ and a group homomorphism $\theta:\bZ_2\to  \textrm{Aut}(G)$.
When there is no danger of ambiguity, 
$\theta(-1)v$ will be denoted by $-v$ for each vertex $v\in V$ and $(-v)(-w)$ will be denoted by $-e$ for each edge $e=vw\in E$.
The action $\theta$ is assumed throughout to be {\em free} on the vertex set of $G$ which means that $v \neq -v$ for all $v\in V$.
It will not be assumed that the action is free on the edge set of $G$ and so there may be edges $e\in E$ such that $e=-e$. Such an edge is said to be {\em fixed} by $\theta$.  
The {\em vertex orbit} of a vertex $v\in V$ under the action $\theta$ is the pair $[v]:=\{v,-v\}$. 
The set of all vertex orbits is denoted $V_0$.
Similarly, the {\em edge orbit} of an edge $e\in E$ is the pair $[e]:=\{e,-e\}$ and the set of all edge orbits is denoted $E_0$. 

\subsection{Symmetric and anti-symmetric motions}
Let $(X,\|\cdot\|)$ be a finite dimensional normed real linear space. A {\em rigid motion} of $(X,\|\cdot\|)$ is a family of continuous paths $\{\alpha_x:[-1,1]\to X\}_{x\in X}$,
such that $\alpha_x(t)$ is differentiable at $t=0$ with $\alpha_x(0)=x$ and $\|\alpha_x(t)-\alpha_y(t)\| =\|x-y\|$ for all pairs $x,y\in X$ and all $t\in [-1,1]$.

The {\em rigidity map} for $G=(V,E)$ and $(X,\|\cdot\|)$ 
is defined by, 
\[f_G:X^{|V|}\to \bR^{|E|},\quad (x_v)_{v\in V}\mapsto (\|x_v-x_w\|)_{vw\in E}.\]
The directional derivative of the rigidity map $f_G$ at a point $p\in X^{|V|}$ and in the direction of a vector $u\in X^{|V|}$ is denoted $D_uf_G(p)$,
\[D_uf_G(p) = \lim_{t\to 0}\,\frac{1}{t}\left(f_G(p+tu)-f_G(p)\right).\]

A {\em bar-joint framework} in $(X,\|\cdot\|)$ is a pair $(G,p)$ where $p=(p_v)_{v\in V}\in X^{|V|}$ and $p_v\not=p_w$ for all $vw\in E$. A {\em subframework} of $(G,p)$ is a bar-joint framework $(H,p_H)$ (or simply $(H,p)$) where $H=(V(H),E(H))$ is a subgraph of $G$ and $p_H=(p_v)_{v\in V(H)}$. A subframework $(H,p)$ is {\em spanning} in $(G,p)$ if $H$ is a spanning subgraph of $G$ and {\em proper} if $H\not=G$. 

An {\em infinitesimal flex} for $(G,p)$ is a vector $u\in X^{|V|}$ such that $D_uf_G(p)=0$. 
The collection of all infinitesimal flexes of $(G,p)$ forms a linear subspace of $X^{|V|}$, denoted $\F(G,p)$. 
It can be shown (see \cite[Lemma 2.1]{kit-pow}) that if $\{\alpha_x\}_{x\in X}$ is a rigid motion of $(X,\|\cdot\|)$ then
$(\alpha_{p_v}'(0))_{v\in V}\in X^{|V|}$ is an infinitesimal flex of $(G,p)$.
An infinitesimal flex of this type is said to be {\em trivial} and the collection of all trivial infinitesimal flexes forms a linear subspace of $\F(G,p)$, denoted  $\T(G,p)$. 
A bar-joint framework is said to be {\em infinitesimally rigid} if every infinitesimal flex is trivial and {\em isostatic} if, in addition, no proper spanning subframework is infinitesimally rigid. 

If the rigidity map $f_G$ is differentiable at $p$ then the differential is denoted $df_G(p)$. In this case, $(G,p)$ is said to be {\em well-positioned} in $(X,\|\cdot\|)$ and $df_G(p)$ is referred to as the {\em rigidity operator} for $(G,p)$.
Note that the rigidity operator $df_G(p)$ satisfies,
\begin{eqnarray}
\label{OpFormula}
df_G(p)u = \left(\,\varphi_{v,w}(u_v-u_w)\,\right)_{vw\in E},
\end{eqnarray}
for all $u=(u_v)_{v\in V}\in X^{|V|}$ where  $\varphi_{v,w}:X\to\bR$ is a linear functional defined by, 
\[\varphi_{v,w}(x) = \lim_{t\to 0}\,
\frac{1}{t}(\|p_v-p_w+tx\|-\|p_v-p_w\|), \,\,\,\,\,\,\, \forall\,x\in X.\] 
In this way the rigidity operator may be represented by a {\em rigidity matrix}  of linear functionals with rows indexed by $E$ and columns indexed by $V$. (For details see \cite{kit-sch}).

Let $\Isom(X,\|\cdot\|)$ denote the group of linear isometries of $(X,\|\cdot\|)$.
A bar-joint framework $(G,p)$ is said to be {\em \textrm{$\bZ_2$-symmetric}} with respect to an action $\theta:\bZ_2\rightarrow \Aut(G)$ and a group representation $\tau:\bZ_2\rightarrow \Isom(X,\|\cdot\|)$  if $\tau(-1)(p_v)=p_{-v}$ for all  $v\in V$.

\begin{lemma} \label{decomp}
\label{Decomp}
Let $(G,p)$ be a well-positioned bar-joint framework in $(X,\|\cdot\|)$ which is $\bZ_2$-symmetric with respect to an action $\theta:\bZ_2\rightarrow \Aut(G)$ and a  representation $\tau:\bZ_2\rightarrow \Isom(X,\|\cdot\|)$.
\begin{enumerate}[label=(\roman*)]
\item
$X^{|V|}$ may be expressed as a direct sum $X^{|V|}=X_1\oplus X_2$ where, 
\begin{eqnarray*}
X_1 &=& \{(x_v)_{v\in V}\in X^{|V|}:\, x_{-v}
= \,\tau(-1)x_v,\,\,\,\,\forall\,\,v\in V\},
\end{eqnarray*}
\begin{eqnarray*}
X_2 &=& \{(x_v)_{v\in V}\in X^{|V|}: x_{-v}
= -\tau(-1)x_v,\,\,\forall\,\,v\in V\}.
\end{eqnarray*}
\item $\bR^{|E|}$ may be expressed as a direct sum 
$\bR^{|E|} = Y_1\oplus Y_2$ where,
\begin{eqnarray*}
Y_1 &=& \{(y_e)_{e\in E}\in \bR^{|E|}:\, y_{-e}=\, y_e,\,\,\,\,
\forall\,\,e\in E\},
\end{eqnarray*}
\begin{eqnarray*}
Y_2 &=& \{(y_e)_{e\in E}\in \bR^{|E|}: y_{-e}=- y_e,\,\,
\forall\,\,e\in E\}.
\end{eqnarray*}
\item
With respect to the direct sum decompositions, 
\[X^{|V|} = X_1\oplus X_2, \,\,\,\,\,\mbox{ and, } \,\,\,\,\,\,
\bR^{|E|} = Y_1\oplus Y_2,\]
the differential $df_G(p)$ may be expressed as a direct sum of linear transformations,
\[df_G(p)=R_1\oplus R_2,\]
where $R_1:X_1\to Y_1$ and $R_2:X_2\to Y_2$.
\end{enumerate}
\end{lemma}

\proof
Each $(x_v)_{v\in V}\in X^{|V|}$ may be expressed 
as a sum $a+b$ where $a=\left(\frac{1}{2}(x_v+\tau(-1)(x_{-v}))\right)_{v\in V}$
and $b=\left(\frac{1}{2}(x_v-\tau(-1)(x_{-v}))\right)_{v\in V}$.
Note that $a\in X_1$ and $b\in X_2$. 
Similarly, each 
$(y_e)_{e\in E}\in \bR^{|E|}$ may be expressed as
a sum $a+b$ where 
$a=\left(\frac{1}{2}(y_e+y_{-e})\right)_{e\in E}\in Y_1$
and $b=\left(\frac{1}{2}(y_e-y_{-e})\right)_{e\in E}\in Y_2$.
To prove $(i)$ and $(ii)$ it only remains to note that $X_1\cap X_2=\{0\}$ and $Y_1\cap Y_2=\{0\}$.

To prove $(iii)$, let $vw\in E$ and note that if $(x_v)_{v\in V}\in X_1$ then,
\[\varphi_{v,w}(x_v-x_w) = \varphi_{-v,-w}(\tau(-1)(x_v-x_w))
=\varphi_{-v,-w}(x_{-v}-x_{-w}).\] 
Similarly, if $(x_v)_{v\in V}\in X_2$ then, 
\[\varphi_{v,w}(x_v-x_w) = \varphi_{-v,-w}(\tau(-1)(x_v-x_w))
=\varphi_{-v,-w}(-(x_{-v}-x_{-w}))=-\varphi_{-v,-w}(x_{-v}-x_{-w}).\]By equation \ref{OpFormula}), $df_G(p)(X_1)\subset Y_1$ and $df_G(p)(X_2)\subset Y_2$ and so the result follows.
\endproof

A vector $u=(u_v)_{v\in V}\in X^{|V|}$ will be called {\em symmetric} if $u\in X_1$ and {\em anti-symmetric} if $u\in X_2$.
The vector spaces of symmetric and anti-symmetric infinitesimal flexes of $(G,p)$ are respectively denoted  $\F_1(G,p)$ and $\F_2(G,p)$. 
Similarly, the vector spaces of symmetric and anti-symmetric trivial infinitesimal flexes are respectively denoted  $\T_1(G,p)$ and $\T_2(G,p)$. A straight-forward verification shows that
$\F(G,p)=\F_1(G,p)\oplus \F_2(G,p)$ 
and $\T(G,p) = \T_1(G,p)\oplus \T_2(G,p)$.

The following observation will be applied in the next section. The identity operator on $X$ is denoted $I$.

\begin{lemma} \label{lem:dimmotions}
Let $(G,p)$ be a well-positioned and $\bZ_2$-symmetric bar-joint framework in $(X,\|\cdot\|)$.
If the group of linear isometries $\Isom(X,\|\cdot\|)$ is finite
then,
\begin{enumerate}[label=(\roman*)]
\item 
$\dim \T(G,p) = \dim X$.
\item 
$\dim \T_1(G,p) = \rank (I+\tau(-1))$. 
\item 
$\dim \T_2(G,p) = \rank (I-\tau(-1))$.
\end{enumerate}
\end{lemma}

\proof
It is shown in \cite{kit-pow} that if $\Isom(X,\|\cdot\|)$ is finite then $\T(G,p) = \{(x,\ldots,x)\in X^{|V|}:x\in X\}$.
Part $(i)$ is an immediate consequence of this while $(ii)$ and $(iii)$ follow on considering the definitions of $X_1$ and $X_2$.  
\endproof

\begin{definition}
A $\bZ_2$-symmetric bar-joint framework $(G,p)$ in $(X,\|\cdot\|)$
 is said to be,
\begin{enumerate}
\item {\em  (anti-)~symmetrically infinitesimally rigid}
if every (anti-)~symmetric infinitesimal flex of $(G,p)$ is a trivial infinitesimal flex.
\item
{\em (anti-)~symmetrically isostatic} if it is (anti-)~symmetrically infinitesimally rigid and no $\bZ_2$-symmetric proper spanning subframework of $(G,p)$ is (anti-)~symmetrically infinitesimally rigid.
\end{enumerate}
\end{definition}

Let $G=(V,E)$ be a $\bZ_2$-symmetric graph with $V_0$ the set of vertex orbits and $E_0$ the set of edge orbits. The subset of $E_0$ consisting of edge orbits for edges in $G$ which are not fixed is denoted $E_0'$.

\begin{lemma}
\label{Counts}
Let $(G,p)$ be a well-positioned and $\bZ_2$-symmetric bar-joint framework in $(X,\|\cdot\|)$.  
\begin{enumerate}[label=(\roman*)]
\item 
If $(G,p)$ is symmetrically infinitesimally rigid then,
\[|E_0| \geq (\dim X)|V_0| - \dim \T_1(G,p).\] 
\item 
If $(G,p)$ is anti-symmetrically infinitesimally rigid then,
\[|E_0'| \geq (\dim X)|V_0| - \dim \T_2(G,p).\]
\end{enumerate}
\end{lemma}

\proof

Consider the decompositions constructed in Lemma \ref{Decomp}.  
Note that $\dim X_1 = (\dim X)|V_0|$, $\dim X_2 = (\dim X)|V_0|$,
$\dim Y_1 = |E_0|$ and $\dim Y_2 = |E_0'|$.
(In the case of $Y_2$ the dimension is determined by the number of edge orbits for edges which are not fixed).
If $(G,p)$ is symmetrically infinitesimally rigid then $\T_1(G,p)=\F_1(G,p)=\ker R_1$ and so,
\[|E_0| \geq \rank R_1  = 
 (\dim X)|V_0| - \dim \ker R_1 = (\dim X)|V_0| - \dim \T_1(G,p).\]
A similar argument applies if $(G,p)$ is anti-symmetrically infinitesimally rigid.
\endproof

\begin{lemma}
\label{AntiSymFixedEdges}
Let $(G,p)$ be a well-positioned and $\bZ_2$-symmetric bar-joint framework in $(X,\|\cdot\|)$.
If $(G,p)$ is anti-symmetrically isostatic then $G$ contains no fixed edges.
\end{lemma}

\proof
Suppose $e=v(-v)$ is a fixed edge in $G$ and let $H=G-e$. 
Then there exists a non-trivial anti-symmetric infinitesimal flex $u\in \F_2(H,p)$. Note that $u\in \ker df_H(p)$ and the linear functional $\varphi_{v,-v}$ satisfies,
\[\varphi_{v,-v}(u_v-u_{-v}) = \varphi_{-v,v}(\tau(-1)(u_v-u_{-v}))
=\varphi_{-v,v}(-(u_{-v}-u_{v}))=-\varphi_{-v,v}(u_{-v}-u_{v})
=-\varphi_{v,-v}(u_{v}-u_{-v}).\]
Thus $\varphi_{v,-v}(u_v-u_{-v})=0$
and so, from equation (\ref{OpFormula}), it follows that $u\in \ker df_G(p)$. In particular, $u$ is a non-trivial anti-symmetric infinitesimal flex of $(G,p)$.
\endproof

Let $Z$ and $W$ be linear subspaces of $X$ such that $X=Z\oplus W$ and suppose $W$ has dimension $1$. A linear isometry $T\in\Isom(X,\|\cdot\|)$ is called a {\em reflection} in the mirror $Z$ along $W$ if $T=I-2P$, where  $P:X\to X$ is the linear projection with range $W$ and kernel $Z$.

\begin{lemma}
\label{K2Lemma}
Let $(K_2,p)$ be a placement of $K_2$ in $(X,\|\cdot\|)$
which is $\bZ_2$-symmetric with respect to an action $\theta:\bZ_2\to \Aut(G)$ and a representation $\tau:\bZ_2\to \Isom(X,\|\cdot\|)$.
If $\theta$ acts freely on $V(K_2)$ and $\tau(-1)$ is a reflection then $(K_2,p)$ is symmetrically isostatic.
\end{lemma}

\proof
Let $v$ and $-v$ be the vertices of $K_2$ and let $u\in \F_1(K_2,p)$ be a symmetric infinitesimal flex of $(K_2,p)$.
The isometry $\tau(-1)$ has the form $\tau(-1)=I-2P$ where $P$ is a projection as described above.
Note that, 
\[\varphi_{v,-v}(Pu_v) =\frac{1}{2}\varphi_{v,-v}((I-\tau(-1))u_v)
= \frac{1}{2}\varphi_{v,-v}(u_v-u_{-v}) =
0.\]  
Thus $u_v\in Z$ or $W\subset \ker\varphi_{v,-v}$.
Note that $p_v-p_{-v} = (I-\tau(-1))p_v =2P(p_v)\in W$.
Thus if $W\subset \ker \varphi_{v,-v}$ then,
\[\|p_v-p_{-v}\| = \varphi_{v,-v}(p_v-p_{-v})=0,\]
and so $p_v=p_{-v}$ which is a contradiction.
We conclude that  $u_v\in Z$ and so $u_{-v}=\tau(-1)u_v=u_v$.
Thus $u$ is a trivial infinitesimal flex.
\endproof

\subsection{Signed quotient graphs} 
\label{sec:basicdef}
The {\em quotient graph} $G_0=G/\bZ_2$  for a $\bZ_2$-symmetric graph $(G,\theta)$ has vertex set $V_0$ consisting of the vertex orbits for $(G,\theta)$ and edge set $E_0$ consisting of the edge orbits. An edge $[e]\in E_0$ is regarded as incident to a vertex $[v]\in V_0$ if $e$ (equivalently, $-e$) is incident to either $v$ or $-v$ in $G$. 
In general, $G_0$ is not a simple graph as if $e\in E$ is a fixed edge in $G$ then $[e]$ is a loop in $G_0$. Also, if $e=vw$ and $e'=v(-w)$ are distinct edges in $G$ then $[e]$ and $[e']$ are parallel edges in $G_0$.

Let $\tilde{V}_0=\{\tilde{v}_1,\ldots,\tilde{v}_n\}$ be a choice of representatives for the vertex orbits of $(G,\theta)$.
A {\em signed quotient graph} (or \emph{quotient $\mathbb{Z}_2$-gain graph} \cite{jkt,schtan}) is a pair $(G_0,\psi)$ consisting of a quotient graph $G_0$ and an edge-labeling (or \emph{gain}) $\psi:E_0\to \bZ_2$ where $\psi([e])=1$ if either $e$ or $-e$ is incident to two vertices in $\tilde{V}_0$ and $\psi([e])=-1$ otherwise. See Figure~\ref{fig:gaingraph} for an example.

In the following, $G$ will be referred to as the \emph{covering graph} of $(G_0,\psi)$ and, to simplify notation, $\psi([e])$ will be denoted $\psi_{[e]}$.
Note that the covering graph is required to be a simple graph and so signed quotient graphs are characterised by the following two properties.
\begin{enumerate}
\item If two edges $[e]$ and $[e']$ in $G_0$ are parallel then $\psi_{[e]}\not=\psi_{[e']}$.
\item If $[e]$ is a loop in $G_0$ then $\psi_{[e]}=-1$.
\end{enumerate}

The {\em gain} of a set of edges $F$ in a signed quotient graph $(G_0,\psi)$ is defined as the product $\psi(F) =\Pi_{[e]\in F} \,\psi_{[e]}$.
A set of edges $F$ is {\em balanced} if it does not contain a cycle of edges, or,  has the property that every cycle of edges in $F$ has gain $1$. A subgraph of $G_0$ is {\em balanced} in $(G_0,\psi)$  if it is spanned by a balanced set of edges, otherwise, the subgraph is {\em unbalanced}. (See also \cite{jkt,zas,zas1}).

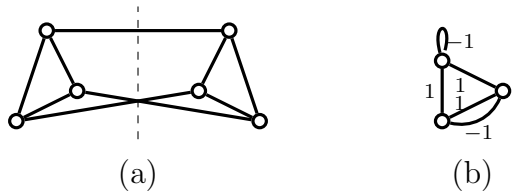
\begin{figure}[htp]
\begin{center}
\begin{tikzpicture}[very thick,scale=0.8]
\tikzstyle{every node}=[circle, draw=black, fill=white, inner sep=0pt, minimum width=5pt];

\node (p1) at (0,0) {};
\node (p2) at (4,0) {};
\node (p1o) at (1,0.5) {};
\node (p2o) at (3,0.5) {};
\node (p1l) at (0.5,1.5) {};
\node (p2r) at (3.5,1.5) {};
\draw(p1)--(p1o);
\draw(p2)--(p2o);
\draw(p1)--(p2o);
\draw(p2)--(p1o);
\draw(p1l)--(p1o);
\draw(p1l)--(p2r);
\draw(p1l)--(p1);
\draw(p2r)--(p2o);
\draw(p2r)--(p2);
\draw[dashed,thin](2,-0.3)--(2,1.9);
   \node [draw=white, fill=white] (a) at (2,-0.9)  {(a)};

 \node [draw=white, fill=white] (a) at (6.8,0.5)  {\tiny $1$};
 \node [draw=white, fill=white] (a) at (7.3,0.6)  {\tiny $1$};

 \node [draw=white, fill=white] (a) at (7.29,0.3)  {\tiny $1$};
 \node [draw=white, fill=white] (a) at (7.3,1.3)  {\tiny $-1$};

 \node [draw=white, fill=white] (a) at (7.6,-0.2)  {\tiny $-1$};

\node(p11) at (7,0) {};
\node(p22) at (7,1) {};
\node (p33) at (8,0.5) {};
\draw(p22)--(p11);
\draw(p22)--(p33);
\draw(p33)--(p11);
\path(p11) edge [bend right=40] (p33);
\path(p22) edge [loop above, >=stealth,shorten >=1pt,looseness=20] (p22);
 \node [draw=white, fill=white] (a) at (7.5,-0.9)  {(b)};
\end{tikzpicture}
\end{center}
\vspace{-0.2cm}
\caption{A $\mathbb{Z}_2$-symmetric graph $(G,\theta)$, where $\theta$ describes the reflectional symmetry  shown in (a) and a corresponding signed quotient graph $(G_0,\psi)$ (b).}
\label{fig:gaingraph}
\end{figure}

\begin{lemma}[\cite{jkt,zas}] \label{switch}
Let $(G_0,\psi)$ be a signed quotient graph for a $\bZ_2$-symmetric graph $(G,\theta)$ and let $H_0$ be a balanced subgraph in $(G_0,\psi)$.
Then, 
\begin{enumerate}[label=(\roman*)]
\item
$H_0$ is a balanced subgraph in $(G_0,\psi')$ for every gain $\psi'$ induced by a choice of vertex orbit representatives for $(G,\theta)$, and,
\item
there exists a choice of vertex orbit representatives $\tilde{V}_0$  for $(G,\theta)$ such that the induced gain $\psi'$ satisfies $\psi'_{[e]}=1$ for all $[e]\in E(H_0)$.
\end{enumerate}
\end{lemma}

A subgraph of $G_0$ will be referred to as balanced if it is balanced in $(G_0,\psi)$ for some (and hence every) gain $\psi$ induced by a choice of vertex orbit representatives.

\begin{definition}
A  subgraph of $G_0$ for which every connected component contains exactly one cycle, each of which is unbalanced, is called an {\em unbalanced  map graph} in $G_0$. 
\end{definition}

If a representative vertex $\tv$ is replaced by the vertex $-\tv$, then  a new  signed quotient graph $(G_0,\psi')$ is obtained, where $\psi'_{[e]}=-\psi_{[e]}$ if $[e]$ is incident with $[v]$, and $\psi'_{[e]}=\psi_{[e]}$ otherwise. 
This is referred to as a \emph{switching operation} on $[v]$.

\subsection{Orbit matrices and sparsity counts}
Let  $(G,p)$ be a well-positioned and $\bZ_2$-symmetric bar-joint framework  in $(X,\|\cdot\|)$ and let $\tilde{V}_0$ be a choice of vertex orbit representatives.

\begin{definition}
A {\em symmetric orbit matrix} for $(G,p)$ is a matrix of linear functionals on $X$, denoted $O_1(G,p)$ or simply $O_1$, with rows indexed by  $E_0$ and columns indexed by $V_0$. 

The matrix entry for a pair $([e],[v])\in E_0\times V_0$ is given by,
\[O_1([e],[v])=
\left\{
\arraycolsep=2pt\def\arraystretch{1.4}
\begin{array}{ll}
\varphi_{\tilde{v},\psi_{[e]}\tilde{w}} 
& \mbox{ if }[e]=[vw] \mbox{ and } [e] \mbox{ is not a loop},\\

2\varphi_{\tilde{v},-\tilde{v}}

& \mbox{ if }[e] \mbox{ is a loop at } [v],\\

0 & \mbox{ otherwise,}
\end{array}\right.\]
where $\tilde{v},\tilde{w}\in \tilde{V}_0$ are the representative vertices for $[v]$ and $[w]$ respectively and  $\psi$ is the gain on $G_0$ induced by $\tilde{V}_0$.
\end{definition}

Each symmetric orbit  matrix determines a linear map $O_1(G,p):X^{|V_0|}\to \bR^{|E_0|}$. Explicitly, the row entries of $O_1(G,p)$ which correspond to an edge orbit $[e]=[vw]$ which is not a loop are,
\[
\kbordermatrix{
& & & & \lbrack v \rbrack & & & & \lbrack w \rbrack & & & \\
\lbrack e \rbrack &  0 &\cdots\,& 0 & \varphi_{\tilde{v},\psi_{[e]}\tilde{w}} &0 &\cdots
\cdots &0&\varphi_{\tilde{w},\psi_{[e]}\tilde{v}}&0&\cdots &0},\]
while if $[e]$ is a loop at a vertex $[v]$ then the row entries are, 
\[\kbordermatrix{ 
&&&&\lbrack v \rbrack &&& \\
\lbrack e \rbrack & 0 &\,\cdots\,& 0& \,\,\,\, 2\varphi_{\tilde{v},-\tilde{v}}
\,\,\,\,&0&\,\cdots\, &0}.\]

\begin{lemma}
\label{CommDiagramLemma1}
Let $(G,p)$ be a well-positioned and $\bZ_2$-symmetric bar-joint framework in $(X,\|\cdot\|)$. If $O_1$ is a symmetric orbit matrix for $(G,p)$ then there exist linear isomorphisms, 
\[S_1:X^{|V_0|}\to X_1, \qquad T_1:\bR^{|E_0|}\to Y_1,\]
such that the following diagram commutes.
\[
\begin{tikzcd}
X^{|V_0|} \arrow{r}{O_1} \arrow[swap]{d}{S_1} & \bR^{|E_0|} \arrow{d}{T_1}  \\
X_1  \arrow{r}{R_1} & Y_1
\end{tikzcd}
\]
In particular, $R_1$ and $O_1$ are (isomorphically) equivalent linear transformations.
\end{lemma}

\proof
Let $\tilde{V}_0$ be the choice of vertex orbit representatives from which $O_1(G,p)$ is derived.
Each vertex $v\in V$ is expressible in the form $v=\gamma_v \tilde{v}$ for some $\gamma_v\in\bZ_2$ where $\tilde{v}\in\tilde{V}_0$ is the chosen representative for $[v]$.
Define,
\[S_1:X^{|V_0|}\to X_1, \,\,\,\,\,\, 
(x_{[v]})_{[v]\in V_0}\mapsto 
(\tau(\gamma_v)x_{[v]})_{v\in V},\]
\[T_1:\bR^{|E_0|}\to Y_1, \,\,\,\,\,\, 
(y_{[e]})_{[e]\in E_0}\mapsto 
(y_{[e]})_{e\in E}.\]
Let $u=(u_{[v]})_{[v]\in V_0}\in X^{|V_0|}$.
It is sufficient to compare the entries of $(T_1\circ O_1)u$ and
$(R_1\circ S_1)u$ in $Y_1$ (note that these entries are indexed by $E$). 

Suppose $e=vw\in E$ is an edge in $G$ which is not fixed. Then the edge orbit $[e]$ is not a loop in the quotient graph $G_0$ 
and so the entry of $O_1(u)$ corresponding to $[e]$ is given by,
\begin{eqnarray*}
\label{RigidityMatrixEqn}
\varphi_{\tilde{v},\psi_{[e]}\tilde{w}}(u_{[v]})-\varphi_{\psi_{[e]}\tilde{v},\tilde{w}}(u_{[w]})
&=& \varphi_{\tilde{v},\psi_{[e]}\tilde{w}}(u_{[v]})-\varphi_{\tilde{v},\psi_{[e]}\tilde{w}}(\tau(\psi_{[e]})u_{[w]}) \\
&=& \varphi_{\tilde{v},\psi_{[e]}\tilde{w}}(u_{[v]} - \tau(\psi_{[e]})u_{[w]}).
\end{eqnarray*}
This is also the entry of $(T_1\circ O_1)u$ corresponding to $e$.
Note that $e=(\gamma_v \tilde{v})(\gamma_w \tilde{w})$ where $\psi_{[e]}=\gamma_v\gamma_w$.
Thus, the entry of $(R_1\circ S_1)u$ corresponding to $e$ is,
\begin{eqnarray*}
\varphi_{v,w}(\tau(\gamma_v)u_{[v]}
-\tau(\gamma_w)u_{[w]})
&=& \varphi_{\gamma_v \tilde{v},\gamma_w \tilde{w}}(\tau(\gamma_v)u_{[v]}
-\tau(\gamma_w)u_{[w]}) \\
&=& \left(\varphi_{\tilde{v},\psi_{[e]}\tilde{w}}\circ\tau(\gamma_v)\right)(\tau(\gamma_v)u_{[v]}
-\tau(\gamma_w)u_{[w]}) \\
&=&\varphi_{\tilde{v},\psi_{[e]}\tilde{w}}(u_{[v]}
-\tau(\psi_{[e]})u_{[w]}).
\end{eqnarray*}

Now suppose $e=\tilde{v}(-\tilde{v})\in E$ is a fixed edge in $G$. The edge orbit $[e]$ is a loop in the quotient graph and so the entry of $(T_1\circ O_1)u$ corresponding to $e$ is $2\varphi_{\tilde{v},-\tilde{v}}(u_{[\tilde{v}]})$. Likewise, the entry of $(R_1\circ S_1)u$ corresponding to $e$ is,
\begin{eqnarray*}
\label{RigidityMatrixEqn2}
\varphi_{\tilde{v},-\tilde{v}}(u_{[\tilde{v}]}-\tau(-1)u_{[\tilde{v}]})
&=& \varphi_{\tilde{v},-\tilde{v}}(u_{[\tilde{v}]})+\varphi_{-\tilde{v},\tilde{v}}(\tau(-1)u_{[\tilde{v}]}) \\
&=& \varphi_{\tilde{v},-\tilde{v}}(u_{[\tilde{v}]}) + \varphi_{\tilde{v},-\tilde{v}}(u_{[\tilde{v}]}) \\
&=& 2\varphi_{\tilde{v},-\tilde{v}}(u_{[\tilde{v}]}).
\end{eqnarray*}
\endproof

Consider again a $\bZ_2$-symmetric bar-joint framework $(G,p)$ and fix an orientation on the edges of the quotient graph which lie in $E_0'$ (i.e. the edges in $G_0$ which are not loops).

\begin{definition}
An {\em anti-symmetric orbit matrix} for $(G,p)$ is a matrix of linear functionals on $X$, denoted $O_2(G,p)$ or $O_2$, with rows indexed by $E_0'$ and columns indexed by $V_0$. 

The matrix entry for a pair $([e],[v])\in E_0'\times V_0$ is given by,
\[O_2([e],[v])=
\left\{
\arraycolsep=2pt\def\arraystretch{1.4}
\begin{array}{ll}
\varphi_{\tilde{v},\psi_{[e]}\tilde{w}} 
& \mbox{ if }[e]=[vw] \mbox{ and } [e] \mbox{ is oriented from } $[v]$ \mbox{ to }$[w]$,\\

\psi_{[e]}\varphi_{\tilde{v},\psi_{[e]}\tilde{w}} 
& \mbox{ if }[e]=[vw] \mbox{ and } [e] \mbox{ is oriented from } $[w]$ \mbox{ to }$[v]$,\\

0 & \mbox{ otherwise,}
\end{array}\right.\]
where $\tilde{v},\tilde{w}\in \tilde{V}_0$ are the representative vertices for $[v]$ and $[w]$ respectively and  $\psi$ is the gain on $G_0$ induced by $\tilde{V}_0$.
\end{definition}

The row entries of $O_2(G,p)$  corresponding to an edge orbit $[e]$  oriented from $[v]$ to $[w]$ are,
\[
\kbordermatrix{
& & & & \lbrack v \rbrack & & & & \lbrack w \rbrack & & & \\
\lbrack e\rbrack& 0 &\cdots& 0& \varphi_{\tilde{v},\psi_{[e]}\tilde{w}}&0&\cdots
\cdots &0&\psi_{[e]}\,\varphi_{\tilde{w},\psi_{[e]}\tilde{v}}
&0&\cdots &0 }.\]

\begin{lemma}
Let $(G,p)$ be a well-positioned and $\bZ_2$-symmetric bar-joint framework in $(X,\|\cdot\|)$. If $O_2$ is an anti-symmetric orbit  matrix for $(G,p)$ then there exist linear isomorphisms, 
\[S_2:X^{|V_0|}\to X_2,\qquad \,T_2:\bR^{|E_0'|}\to Y_2\] 
such that the following diagram commutes.
\[
\begin{tikzcd}
X^{|V_0|} \arrow{r}{O_2} \arrow[swap]{d}{S_2} & \bR^{|E_0'|} \arrow{d}{T_2}  \\
X_2  \arrow{r}{R_2} & Y_2
\end{tikzcd}
\]
In particular, $R_2$ and $O_2$ are (isomorphically) equivalent linear transformations.
\end{lemma}

\proof
Each vertex $v\in V$ is expressible in the form $v=\gamma_v \tilde{v}$ for some $\gamma_v\in\bZ_2$ where $\tilde{v}\in\tilde{V}_0$ is the chosen representative for $[v]$.
For each edge $e=vw\in E$ which is not fixed, 
define $\gamma_e=\gamma_v$ if $[e]$ is oriented from $[v]$ to $[w]$. Also define,
\[S_2:X^{|V_0|}\to X_2, \,\,\,\,\,\, 
(x_{[v]})_{[v]\in V_0}\mapsto 
(\gamma_v\tau(\gamma_v)x_{[v]})_{v\in V},\]
\[T_2:\bR^{|E_0'|}\to Y_2, \,\,\,\,\,\, 
(y_{[e]})_{[e]\in E_0'}\mapsto 
(\gamma_e y_{[e]})_{e\in E},\]
where, in the definition of $T_2$,  we formally set $\gamma_e y_{[e]}=0$ if $e$ is a fixed edge of $G$. 
The commutativity of the diagram can now be verified in a manner analogous to the proof of Lemma \ref{CommDiagramLemma1}.
\endproof

Let $(H,p)$ be a $\bZ_2$-symmetric framework.
If $H_0$ is balanced then, by Lemma \ref{switch}, there exists a  choice of vertex orbit representatives $\tilde{V}_0$ such that the induced gain is identically $1$ on the edges of $H_0$.
It follows that $H_0$ may be identified with the vertex-induced subgraph on $\tilde{V}_0$ in $H$.
With this identification, $(H_0,p)$ is a well-defined subframework of $(H,p)$.

\begin{lemma} \label{lem:symcounts}
Let $(G,p)$ be a well-positioned and $\bZ_2$-symmetric bar-joint framework in $(X,\|\cdot\|)$ and let $(H,p)$ be a $\bZ_2$-symmetric subframework of $G$.  
\begin{enumerate}[label=(\roman*)]
\item 
If $(G,p)$ is symmetrically isostatic then,
\[|E(H_0)| \leq (\dim X)|V(H_0)| - \dim \T_1(H,p),\]
and if $H_0$ is balanced in $G_0$ then,
\[|E(H_0)| \leq (\dim X)|V(H_0)| - \dim \T(H_0,p).\]
\item 
If $(G,p)$ is anti-symmetrically isostatic then,
\[|E(H_0)| \leq (\dim X)|V(H_0)| - \dim \T_2(H,p),\]
and if $H_0$ is balanced in $G_0$ then,
\[|E(H_0)| \leq (\dim X)|V(H_0)| - \dim \T(H_0,p).\]
\end{enumerate}
\end{lemma}

\proof
By Lemma \ref{CommDiagramLemma1}, if $(G,p)$ is symmetrically isostatic then $O_1(H,p)$ is row independent and,
\[|E(H_0)| =\rank O_1(H,p)
= (\dim X)|V(H_0)| - \dim \ker O_1(H,p) 
\leq (\dim X)|V(H_0)|-\dim \T_1(H,p).\]
If $H_0$ is balanced then for some choice of vertex orbit representatives each edge of $H_0$ has gain $1$. By the remark preceding the lemma, $(H_0,p)$ is a well-positioned framework in $(X,\|\cdot\|)$ and, by equation (\ref{OpFormula}),  
$df_{H_0}(p) = O_1(H,p)$.
Thus, 
\[|E(H_0)|=\rank O_1(H,p)=\rank df_{H_0}(p)
\leq (\dim X)|V(H_0)| - \dim \T(H_0,p).\]  
This proves $(i)$ and the proof of $(ii)$ is similar.
 
\endproof

\section{Grid-like frameworks with reflectional symmetry}
\label{Sect:Grid}
In this section we consider bar-joint frameworks in $(\bR^2, \|\cdot\|_\P)$ where the norm $\|\cdot\|_\P$ has the property that the closed unit ball $\P=\{x\in \bR^2: \|x\|_\P\leq 1\}$ is a quadrilateral.
(The $\ell^1$ and $\ell^\infty$ norms are familiar examples of such norms. In general, every absolutely convex quadrilateral is the closed unit ball for a unique norm on $\bR^2$ defined by the Minkowski functional for the quadrilateral).
The norm is expressed by the formula,
\[\|x\|_\P = \max_{j=1,2}\, |\hat{F}_j\cdot x|, \,\,\,\,\,\,\, \forall x\in \bR^2,\]
where $\P = \bigcap_{j=1,2} \,\{x\in\bR^2: |x\cdot\hat{F}_j|\leq 1\}$.
Note that the boundary of $\P$ consists of four facets $\pm F_1$, $\pm F_2$ and that for each $j=1,2$, $\hat{F}_j$ is the unique extreme point of the polar set of $\P$ for which $F_j=\{x\in \P:\hat{F}_j\cdot x =1\}$.
Also note that each facet $F_j$ determines a linear functional, 
\[\varphi_{F_j}:X\to\bR,\quad x\mapsto \hat{F}_{j}\cdot x.\]

\subsection{Monochrome subgraph decompositions}
\label{Sect:framework}
Let $(G,p)$ be a bar-joint framework in $(\bR^2,\|\cdot\|_\P)$ and let $F$ be a facet of $\P$.
An edge $vw\in E$ is said to have the induced {\em framework colour} $[F]$ if $p_v-p_w$ is contained in the cone of $F$ or $-F$.
The subgraph of $G$ spanned by edges with framework colour $[F]$ is denoted by $G_{F}$ and referred to as an induced {\em monochrome subgraph} of $G$. 
Note that if $(G,p)$ is well-positioned then each edge $vw$ has exactly one framework colour $[F]$ and the linear functional $\varphi_{v,w}$ is given by either $\varphi_{F}$ or $\varphi_{-F}$. 
The following result was obtained (for $d$-dimensional frameworks) in \cite{kit-pow}.

\begin{theorem}
Let $(G,p)$ be a well-positioned bar-joint framework
in $(\bR^2, \|\cdot\|_\P)$.
Then $(G,p)$ is isostatic if and only if the monochrome subgraphs $G_{F_1}$ and $G_{F_2}$ are both spanning trees in $G$. 
\end{theorem}

We will now prove symmetric analogues of the above theorem for frameworks with reflectional symmetry.
Let $(G,p)$ be $\bZ_2$-symmetric with respect to $\theta:\bZ_2\to \Aut(G)$ and $\tau:\bZ_2\to \Isom(\mathbb{R}^2,\|\cdot\|_\P)$ where
$\tau(-1)$ is a reflection in the mirror $\ker\varphi_{F_1}$ along $\ker \varphi_{F_2}$. 
Then for each edge $e\in E$, both $e$ and $-e$ have the same induced framework colour and this will be referred to as the framework colour of the edge orbit $[e]$.
Define $G_{F,0}$ to be the {\em monochrome subgraph} of the quotient graph $G_0$ spanned by edges $[e]$ with framework colour $[F]$.

In the following, the set of vertex orbit representatives for $G$ will be denoted by $\tilde{V}_0=\{\tv_1,\ldots, \tv_n\}$ and $\tilde{V}_1$ will denote the set $\{-\tv_1,\ldots, -\tv_n\}$.

\begin{theorem}[Symmetrically isostatic frameworks]
\label{thm:forced1}
Let $(G,p)$ be a well-positioned and $\mathbb{Z}_2$-symmetric bar-joint framework in $(\bR^2,\|\cdot\|_\P)$ where $\P$ is a quadrilateral and $G\not=K_2$.
Suppose $\theta$ acts freely on $V$ and $\tau(-1)$ is  a  reflection in the mirror $\ker\varphi_{F_1}$ along $\ker \varphi_{F_2}$. 
The following  are equivalent.
\begin{enumerate}
\item[(i)] $(G,p)$ is  symmetrically isostatic.
\item[(ii)]
$G_{F_1,0}$ is a spanning unbalanced map graph in $G_0$ and 
$G_{F_2,0}$ is a spanning tree in $G_0$.
\end{enumerate}
\end{theorem}

\begin{proof}
$(i)\Rightarrow (ii)$
Suppose there exists a vertex $[v_0]\in V_0 \setminus V(G_{F_1,0})$. 
Choose a non-zero vector $x\in \ker\varphi_{F_2}$ and for all $v\in V$ define,
\[u_{v} = \left\{ \begin{array}{ll}
x & \mbox{ if } v=\tv_0, \\
-x & \mbox{ if } v=-\tv_0, \\
0 & \mbox{ otherwise. } \\
\end{array}\right.
\]
Then $u$ is a non-trivial  symmetric infinitesimal flex for $(G,p)$.
Similarly, if there exists a vertex $[v_0]\in V_0 \setminus V(G_{F_2,0})$ 
then choose a non-zero vector $x\in \ker\varphi_{F_1}$. For all $v\in V$ define,
\[u_{v} = \left\{ \begin{array}{ll}
x & \mbox{ if } [v]=[v_0], \\
0 & \mbox{ otherwise. } \\
\end{array}\right.
\]
Again, $u$ is a non-trivial  symmetric infinitesimal flex for $(G,p)$.
In each case we obtained a contradiction and so $G_{F_2,0}$ and $G_{F_2,0}$ are both spanning subgraphs of $G_0$.

Suppose $G_{F_1,0}$ has a connected component $H_0$ which is a balanced subgraph of $G_0$.
Then by Lemma \ref{switch}, by applying switching operations if necessary, we may assume each edge of $H_0$ has trivial gain. 
Thus, if $H$ is the covering graph for $H_0$, then there is no edge $vw\in E(H)$ with $v\in \tilde{V}_0$ and $w\in \tilde{V}_1$.
Choose a non-zero vector $x\in \ker\varphi_{F_2}$ and for all 
$v\in V$ define,
\[u_{v} = \left\{ \begin{array}{ll}
x & \mbox{ if } [v]\in V(H_0) \mbox{ and } v\in \tilde{V}_0,   \\
-x & \mbox{ if } [v]\in V(H_0) \mbox{ and } v\in \tilde{V}_1,   \\
0 & \mbox{ otherwise. }
\end{array}\right.
\]
Then $u$ is a non-trivial  symmetric infinitesimal flex for $(G,p)$ which is a contradiction. 
Thus each connected component of $G_{F_1,0}$ is an unbalanced subgraph of $G_0$.

Suppose $G_{F_2,0}$ is not connected, and 
let $H_0$ be a connected component of $G_{F_2,0}$.
Choose a non-zero vector $x\in \ker\varphi_{F_1}$ and for all 
$v\in V$ define,
\[u_{v} = \left\{ \begin{array}{ll}
x & \mbox{ if } [v]\in V(H_0),  \\
0 & \mbox{ otherwise. } \\
\end{array}\right.
\]
Again $u$ is a non-trivial  symmetric infinitesimal flex for $(G,p)$ and this is a contradiction. Thus $G_{F_2,0}$ is a connected  spanning subgraph of $G_0$.

By Lemma \ref{lem:dimmotions},  $\dim\T_1(G,p)=\rank (I+\tau(-1))=1$.
Thus by Lemmas \ref{Counts} and \ref{lem:symcounts}, $|E_0|=2|V_0|-1$.
Note that each connected component of $G_{F_1,0}$ must contain a cycle (since it is unbalanced) and so if $G_{F_1,0}$ has $n$ connected components, $H_1,H_2,\ldots, H_n$ say, then $|E(H_j)|\geq|V(H_j)|$ for each $j$ and,
\[|E(G_{F_1,0})|=\sum_{j=1}^n |E(H_j)|\geq\sum_{j=1}^n|V(H_j)|= |V_0|.\]
Since $G_{F_2,0}$ is connected it must contain a spanning tree and so 
$|E(G_{F_2,0})|\geq |V_0|-1$.
It follows that $|E(G_{F_1,0})|= |V_0|$, 
$|E(G_{F_2,0})|=|V_0|-1$ and $|E(H_j)|=|V(H_j)|$ for each $j$.
Thus $G_{F_1,0}$ is an unbalanced spanning map graph and $G_{F_2,0}$ is a spanning tree in $G_0$.

$(ii)\Rightarrow (i)$
Suppose $(ii)$ holds and let $u$ be a  symmetric infinitesimal flex of $(G,p)$. 
Let $v\in V$ and note that since $G_{F_1,0}$ has a unique unbalanced cycle, the covering graph for $H_0$ is a connected subgraph of $G_{F_1}$ which contains both $v$ and $-v$. 
In particular, there is a path $vv_1,v_1v_2,\ldots,v_n(-v)$ in $G_{F_1}$ from $v$ to $-v$ and so, 
\[u_{v}-u_{-v} = 
(u_v-u_{v_1})+(u_{v_1}-u_{v_2})+\cdots+(u_{v_n}-u_{-v})
 \in\ker\varphi_{F_1}.\]
Also note that  
$u_v-u_{-v}=(I-\tau(-1))u_{v}=2Pu_v\in \ker\varphi_{F_2}$.
Thus $u_{v}=u_{-v}$ for all $v\in V$.
Since $u_{-v}=\tau(-1)u_{v}$ it also follows that $u_{v}\in\ker\varphi_{F_1}$ for all $v\in V$.
Let $e=vw\in E$.
It is clear  that $u_{v}-u_{w}\in \ker \varphi_{F_1}$.
Since $G_{F_2,0}$ is a spanning tree in $G_0$ there exists a path 
 in $G_{F_2,0}$ from $[v]$ to $[w]$ with  gain $\gamma'$ say. 
Thus there exists a path in $G_{F_2}$ from $v$ to $\gamma' w$
and so $u_v-u_w=u_{v}-u_{\gamma' w}\in \ker \varphi_{F_2}$.
We conclude that $u_{v}=u_{w}$ for all $vw\in E$ and so $u$ is a trivial infinitesimal flex of $(G,p)$.
To see that $(G,p)$ is symmetrically isostatic note that 
$|E_0|= 2|V_0|-1$ and apply Lemma \ref{Counts}.
\end{proof}

The following theorem characterises anti-symmetric isostatic frameworks and is a counterpart to the previous theorem. While the statement and proof are similar there are some key differences. In particular, the roles of the monochrome subgraphs are reversed.

\begin{theorem}[Anti-symmetrically isostatic frameworks]\label{thm:antisym1}
Let $(G,p)$ be a well-positioned and $\mathbb{Z}_2$-symmetric bar-joint framework in $(\bR^2,\|\cdot\|_\P)$ where $\P$ is a quadrilateral.
Suppose $\theta$ acts freely on $V$ and  $\tau(-1)$ is  a  reflection in the mirror $\ker\varphi_{F_1}$ along $\ker \varphi_{F_2}$.
The following  are equivalent.
\begin{enumerate}
\item[(i)] $(G,p)$ is  anti-symmetrically isostatic.
\item[(ii)]
$G_{F_1,0}$ is a spanning tree in $G_0$ and
$G_{F_2,0}$ is a spanning unbalanced map graph in $G_0$.
\end{enumerate}
\end{theorem}

\begin{proof}
$(i)\Rightarrow (ii)$
Suppose  there exists a vertex $[v_0]\in V_0 \setminus V(G_{F_1,0})$.
Choose a non-zero vector $x\in \ker\varphi_{F_2}$. For all $v\in V$ define,
\[u_{v} = \left\{ \begin{array}{ll}
x & \mbox{ if } [v]=[v_0],\\
0 & \mbox{ otherwise. } \\
\end{array}\right.
\]
Similarly, suppose there exists a vertex $[v_0]\in V_0 \setminus V(G_{F_2,0})$. 
Choose a non-zero vector $x\in \ker\varphi_{F_1}$ and for all $v\in V$ define,
\[u_{v} = \left\{ \begin{array}{ll}
x & \mbox{ if } v=\tv_0,  \\
-x & \mbox{ if } v=-\tv_0,  \\
0 & \mbox{ otherwise. } \\
\end{array}\right.
\]
In each case $u$ is a non-trivial anti-symmetric infinitesimal flex for $(G,p)$.

Suppose $G_{F_2,0}$ has a connected component $H_0$ which is a balanced subgraph of $G_0$.
Then, using some switching operations if necessary, we may assume $H_0$ has trivial gain. Choose a non-zero vector $x\in \ker\varphi_{F_1}$ and for all 
$v\in V$ define,
\[u_{v} = \left\{ \begin{array}{ll}
x & \mbox{ if } [v]\in V(H_0) \mbox{ and } v\in \tilde{V}_0,\\
-x & \mbox{ if } [v]\in V(H_0) \mbox{ and } v\in \tilde{V}_1,  \\
0 & \mbox{ otherwise. }
\end{array}\right.
\]
Similarly, suppose $G_{F_1,0}$ is not connected, 
and let $H_0$ be a connected component of $G_{F_1,0}$. Choose a non-zero vector $x\in \ker\varphi_{F_2}$ and for all 
$v\in V$ define, 
\[u_{v} = \left\{ \begin{array}{ll}
x & \mbox{ if } [v]\in V(H_0) ,\\
0 & \mbox{ otherwise. }
\end{array}\right.
\]
Again, in each case $u$ is a non-trivial anti-symmetric infinitesimal flex for $(G,p)$. 
The remainder of the proof is similar to Theorem \ref{thm:forced1}.

$(ii)\Rightarrow (i)$
Apply an argument as in Theorem \ref{thm:forced1} but with the roles of $G_{F_1,0}$ and $G_{F_2,0}$ reversed. 
\end{proof}

The previous two theorems can be combined to obtain the following  characterisation of general infinitesimal rigidity, again expressed in terms of monochrome subgraph decompositions in the quotient graph.

\begin{corollary}[Infinitesimally rigid frameworks] \label{cor}
Let $(G,p)$ be a well-positioned and $\mathbb{Z}_2$-symmetric bar-joint framework in $(\bR^2,\|\cdot\|_\P)$ where $\P$ is a quadrilateral.
Suppose $\theta$ acts freely on $V$ and  $\tau(-1)$ is  a  reflection in the mirror $\ker\varphi_{F_1}$ along $\ker \varphi_{F_2}$.  
The following  are equivalent.
\begin{enumerate}
\item[(i)] $(G,p)$ is infinitesimally rigid.
\item[(ii)] The monochrome subgraphs of $G_0$ both contain connected spanning unbalanced map graphs.
\end{enumerate}
\end{corollary}

\begin{proof}
$(i)\Rightarrow (ii)$
If $(G,p)$ is infinitesimally rigid then it is both symmetrically and anti-symmetrically infinitesimally  rigid.
By removing edge orbits from $G$ we arrive at a $\bZ_2$-symmetric spanning subgraph $A$ such that $(A,p)$ is  symmetrically isostatic.
By Theorem \ref{thm:forced1},
$A_{F_1,0}$ is a spanning unbalanced  map graph in $G_0$ and $A_{F_2,0}$ is a spanning tree.
Similarly, by removing edge orbits from $G$ we arrive at a $\bZ_2$-symmetric spanning subgraph $B$ such that $(B,p)$ is  anti-symmetrically isostatic.
By Theorem \ref{thm:antisym1}, $B_{F_1,0}$ is a spanning tree and
$B_{F_2,0}$ is a spanning unbalanced map graph in $G_0$.
Let $C_1,\ldots,C_n$ be the connected components of $A_{F_1,0}$.
For $k=2,\ldots, n$, let $T_k$ be a spanning tree for $C_k$. 
Note that there exist edges $[e_1],\ldots,[e_{n-1}]$ in the spanning tree $B_{F_1,0}$ such that $C_1\cup T_2\cup\cdots\cup T_n \cup\{[e_1],\ldots,[e_{n-1}]\}$ is a connected spanning unbalanced map graph in $G_{F_1,0}$. A similar argument shows that $G_{F_2,0}$ contains a connected spanning unbalanced map graph.

$(ii)\Rightarrow (i)$
Suppose $(ii)$ holds.
Let $H_{F_1,0}$ and $H_{F_2,0}$ be connected spanning unbalanced map graphs in $G_0$.
Note that $H_{F_2,0}$ contains a spanning tree for $G_0$ and so,
by Theorem \ref{thm:forced1}, $(G,p)$ is symmetrically infinitesimally rigid.
Similarly, $H_{F_1,0}$ contains a spanning tree for $G_0$ and so, by Theorem \ref{thm:antisym1}, $(G,p)$ is anti-symmetrically infinitesimally rigid.
Hence $(G,p)$ is infinitesimally rigid. 

\end{proof}

\subsection{Existence of rigid placements with reflectional symmetry} 
\label{Sect:graph}

In this section, necessary and sufficient conditions are obtained for a $\mathbb{Z}_2$-symmetric graph to have a well-positioned  symmetric or anti-symmetric infinitesimally rigid realisation as a grid-like reflection framework.
A signed quotient graph $(G_0,\psi)$ is  \emph{$(2,2,1)$-gain-sparse} if it satisfies
\begin{itemize}
\item[(i)] $|F|\leq 2|V(F)|-2$ for every balanced $F\subseteq E_0$;
\item[(ii)] $|F|\leq 2|V(F)|-1$ for every  $F\subseteq E_0$.
\end{itemize}
If, in addition,  $|E_0|= 2|V_0|-1$, then $(G_0,\psi)$ is said to be \emph{$(2,2,1)$-gain-tight}. 

We will now describe a number of recursive operations on a $(2,2,1)$-gain tight signed quotient graph $(G_0,\psi)$. See also \cite{jkt,anbs,schtan} for a description of some of these moves. 

\begin{definition} A \emph{Henneberg 1 move} is an addition of a new vertex $[v]$ and two new edges $[e_1]$ and $[e_2]$ to $(G_0,\psi)$, where $[e_1]$ and $[e_2]$ are incident with $[v]$ and are not both loops at $[v]$. If $[e_1]$ and $[e_2]$ are parallel edges, then the gain labels are assigned so that $\psi_{[e_1]}\neq \psi_{[e_2]}$.
\end{definition}
If $[e_1]$ and $[e_2]$ are non-parallel and neither is a loop then the move is called $H1a$. If these edges are parallel the move is called $H1b$. If one of the edges is a loop, then the move is called $H1c$. See also  Figure~\ref{fig:inductive}.

\begin{figure}[htp]
\begin{center}
\begin{tikzpicture}[very thick,scale=0.8]
\tikzstyle{every node}=[circle, draw=black, fill=white, inner sep=0pt, minimum width=5pt];
\filldraw[fill=black!20!white, draw=black, thin, dashed](0,0)circle(0.8cm);
\node (p1) at (40:0.5cm) {};
\node (p2) at (140:0.5cm) {};
   \node [draw=white, fill=white] (a) at (1.7,-0.42)  {(a)};
  \draw[ultra thick,->](1.5,0.1)--(2,0.1);

\end{tikzpicture}
\hspace{0.1cm}
\begin{tikzpicture}[very thick,scale=0.8]
\tikzstyle{every node}=[circle, draw=black, fill=white, inner sep=0pt, minimum width=5pt];
\filldraw[fill=black!20!white, draw=black, thin, dashed](0,0)circle(0.8cm);
\node (p1) at (40:0.5cm) {};
\node (p2) at (140:0.5cm) {};
\node (p3) at (90:1.2cm) {};
\draw[gray] (p3)--(p1);
\draw (p3)--(p2);
\end{tikzpicture}
\hspace{1.2cm}
\begin{tikzpicture}[very thick,scale=0.8]
\tikzstyle{every node}=[circle, draw=black, fill=white, inner sep=0pt, minimum width=5pt];
\filldraw[fill=black!20!white, draw=black, thin, dashed](0,0)circle(0.8cm);
\node (p4) at (90:0.2cm) {};
\node [draw=white, fill=white] (a) at (1.7,-0.42)  {(b)};
\draw[ultra thick, ->](1.5,0.1)--(2,0.1);
   \end{tikzpicture}
\hspace{0.1cm}   \begin{tikzpicture}[very thick,scale=0.8]
\tikzstyle{every node}=[circle, draw=black, fill=white, inner sep=0pt, minimum width=5pt];
\filldraw[fill=black!20!white, draw=black, thin, dashed](0,0)circle(0.8cm);
\node (p3) at (90:1.2cm) {};
\node (p4) at (90:0.2cm) {};
\path
(p3) edge [bend right=22] (p4);
\path[gray]
(p3) edge [bend left=22] (p4);
\end{tikzpicture}
\hspace{1.2cm}
\begin{tikzpicture}[very thick,scale=0.8]
\tikzstyle{every node}=[circle, draw=black, fill=white, inner sep=0pt, minimum width=5pt];
\filldraw[fill=black!20!white, draw=black, thin, dashed](0,0)circle(0.8cm);
\node (p4) at (90:0.2cm) {};
   \node [draw=white, fill=white] (a) at (1.7,-0.42)  {(c)};
\draw[ultra thick, ->](1.5,0.1)--(2,0.1);
\end{tikzpicture}
\hspace{0.1cm}   \begin{tikzpicture}[very thick,scale=0.8]
\tikzstyle{every node}=[circle, draw=black, fill=white, inner sep=0pt, minimum width=5pt];
\filldraw[fill=black!20!white, draw=black, thin, dashed](0,0)circle(0.8cm);
\node (p3) at (90:1.2cm) {};
\node (p4) at (90:0.2cm) {};
\draw (p3)--(p4);
\path[gray]
(p3) edge [loop above, >=stealth,shorten >=1pt,looseness=20] (p3);
\end{tikzpicture}
\end{center}
\vspace{-0.2cm}
\caption{Henneberg 1 moves (with gain labels of edges omitted): (a) H1a-move;
(b) H1b-move;
(c) H1c-move.  }
\label{fig:inductive}
\end{figure}
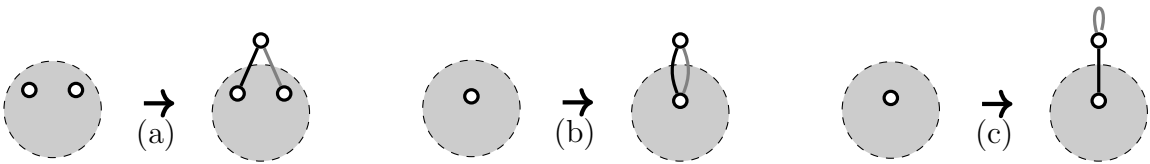

\begin{definition} A \emph{Henneberg 2 move} deletes an edge $[e]$ of $(G_0,\psi)$ and adds a new vertex $[v]$ of degree $3$ to $(G_0,\psi)$ as follows.
The edge $[e]$ is subdivided into two new edges $[e_1]$ and $[e_2]$ (both incident with $[v]$) so that the gains of the new edges satisfy $\psi_{[e_1]}\cdot \psi_{[e_2]}=\psi_{[e]}$. 
Finally, the third new edge, $[e_3]$, joins $[v]$ to a vertex $[z]$ of $(G_0,\psi)$ so that every 2-cycle $[e_i][e_j]$, if it exists, is unbalanced. 
\end{definition}
Suppose first that the edge $[e]$ is not a loop. If none of the edges $[e_i]$ are  parallel, then the move is called H2a. If two of the edges $[e_i]$ are parallel (i.e., $[z]$ is an end-vertex of $[e]$), then the move is called H2b. If the edge $[e]$ is a loop, then the move is called H2c.  See Figure~\ref{fig:inductiveH2}.

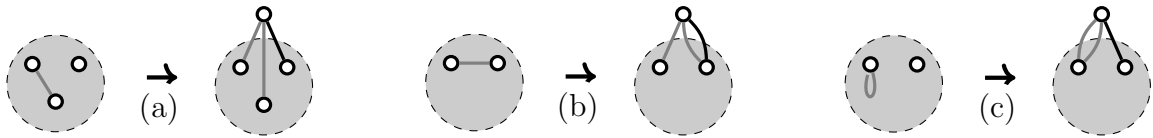
\begin{figure}[htp]
\begin{center}
\begin{tikzpicture}[very thick,scale=0.8]
\tikzstyle{every node}=[circle, draw=black, fill=white, inner sep=0pt, minimum width=5pt];
\filldraw[fill=black!20!white, draw=black, thin, dashed](0,0)circle(0.8cm);
\node (p1) at (40:0.5cm) {};
\node (p2) at (140:0.5cm) {};
\node (p3) at (270:0.3cm) {};
\draw[gray](p2)--(p3);
  \draw[ultra thick,->](1.5,0.1)--(2,0.1);
   \node [draw=white, fill=white] (a) at (1.7,-0.42)  {(a)};
\end{tikzpicture}
\hspace{0.1cm}
\begin{tikzpicture}[very thick,scale=0.8]
\tikzstyle{every node}=[circle, draw=black, fill=white, inner sep=0pt, minimum width=5pt];
\filldraw[fill=black!20!white, draw=black, thin, dashed](0,0)circle(0.8cm);
\node (p1) at (40:0.5cm) {};
\node (p2) at (140:0.5cm) {};
\node (p3) at (270:0.3cm) {};
\node (p4) at (90:1.2cm) {};
\draw(p4)--(p1);
\draw[gray]  (p4)--(p2);
\draw[gray]  (p4)--(p3);
\end{tikzpicture}
\hspace{1.2cm}
\begin{tikzpicture}[very thick,scale=0.8]
\tikzstyle{every node}=[circle, draw=black, fill=white, inner sep=0pt, minimum width=5pt];
\filldraw[fill=black!20!white, draw=black, thin, dashed](0,0)circle(0.8cm);
\node (p1) at (40:0.5cm) {};
\node (p2) at (140:0.5cm) {};
\draw[gray](p2)--(p1);
\draw[ultra thick, ->](1.5,0.1)--(2,0.1);
   \node [draw=white, fill=white] (a) at (1.7,-0.42)  {(b)};
\end{tikzpicture}
\hspace{0.1cm}   \begin{tikzpicture}[very thick,scale=0.8]
\tikzstyle{every node}=[circle, draw=black, fill=white, inner sep=0pt, minimum width=5pt];
\filldraw[fill=black!20!white, draw=black, thin, dashed](0,0)circle(0.8cm);
\node (p3) at (90:1.2cm) {};
\node (p1) at (40:0.5cm) {};
\node (p2) at (140:0.5cm) {};
\draw[gray](p3)--(p2);
\path
(p1) edge [bend right=22] (p3);
\path[gray]
(p1) edge [bend left=22] (p3);
\end{tikzpicture}
\hspace{1.2cm}
\begin{tikzpicture}[very thick,scale=0.8]
\tikzstyle{every node}=[circle, draw=black, fill=white, inner sep=0pt, minimum width=5pt];
\filldraw[fill=black!20!white, draw=black, thin, dashed](0,0)circle(0.8cm);
\node (p1) at (40:0.5cm) {};
\node (p2) at (140:0.5cm) {};
\path[gray]
(p2) edge [loop below, >=stealth,shorten >=1pt,looseness=20] (p2);
\draw[ultra thick, ->](1.5,0.1)--(2,0.1);
   \node [draw=white, fill=white] (a) at (1.7,-0.42)  {(c)};
\end{tikzpicture}
\hspace{0.1cm}   \begin{tikzpicture}[very thick,scale=0.8]
\tikzstyle{every node}=[circle, draw=black, fill=white, inner sep=0pt, minimum width=5pt];
\filldraw[fill=black!20!white, draw=black, thin, dashed](0,0)circle(0.8cm);
\node (p3) at (90:1.2cm) {};
\node (p1) at (40:0.5cm) {};
\node (p2) at (140:0.5cm) {};
\path[gray] (p2) edge [bend left=22] (p3);
\path
(p1) edge [bend right=0] (p3);
\path[gray]
(p2) edge [bend right=22] (p3);
\end{tikzpicture}
\end{center}
\vspace{-0.2cm}
\caption{Henneberg 2 moves (with gain labels of edges omitted): (a) H2a-move;
(b) H2b-move;
(c) H2c-move.  }
\label{fig:inductiveH2}
\end{figure}

\begin{definition}
A \emph{vertex-to-$K_4$} move removes a vertex $[v]$ (of arbitrary degree) and all the edges incident with $[v]$, and adds in a copy of $K_4$ with only trivial gains.
Each removed edge $[x][v]$ is replaced by an edge $[x][y]$ for some $[y]$ in the new $K_4$, where the gain is preserved. If the deleted vertex $[v]$ is incident to a loop, then this loop is replaced by an edge $[y][z]$ with gain $-1$, where $[y]$ and $[z]$ are two (not necessarily distinct) vertices of the new $K_4$. 
\end{definition}
See Figure~\ref{fig:inductiveH3}(a).

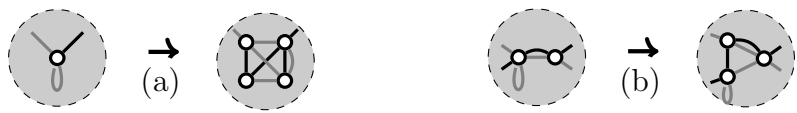
\begin{figure}[htp]
\begin{center}
\begin{tikzpicture}[very thick,scale=0.8]
\tikzstyle{every node}=[circle, draw=black, fill=white, inner sep=0pt, minimum width=5pt];
\filldraw[fill=black!20!white, draw=black, thin, dashed](0,0)circle(0.8cm);
\node (p1) at (0,0) {};
\draw(p1)--(45:0.6cm);
\draw[gray](p1)--(135:0.6cm);
\path[gray]
(p1) edge [loop below, >=stealth,shorten >=1pt,looseness=20] (p1);
\draw[ultra thick, ->](1.5,0.1)--(2,0.1);
   \node [draw=white, fill=white] (a) at (1.7,-0.42)  {(a)};
\end{tikzpicture}
\hspace{0.1cm}   \begin{tikzpicture}[very thick,scale=0.8]
\tikzstyle{every node}=[circle, draw=black, fill=white, inner sep=0pt, minimum width=5pt];
\filldraw[fill=black!20!white, draw=black, thin, dashed](0,0)circle(0.8cm);
\node (p1) at (45:0.45cm) {};
\node (p2) at (135:0.45cm) {};
\node (p3) at (225:0.45cm) {};
\node (p4) at (315:0.45cm) {};
\draw(p1)--(45:0.75cm);
\draw[gray](p2)--(135:0.75cm);

\draw[gray](p1)--(p2);
\draw(p3)--(p2);
\draw[gray](p3)--(p4);
\draw(p1)--(p4);
\draw(p1)--(p3);
\draw[gray](p4)--(p2);
\path[gray]
(p1) edge [bend left=30] (p4);
\end{tikzpicture}
\hspace{2cm}
\begin{tikzpicture}[very thick,scale=0.8]
\tikzstyle{every node}=[circle, draw=black, fill=white, inner sep=0pt, minimum width=5pt];
\filldraw[fill=black!20!white, draw=black, thin, dashed](0,0)circle(0.8cm);
\node (p1) at (-0.3,0) {};
\node (p2) at (0.3,0) {};
\draw[gray] (p1)--(p2);
\draw[gray](p1)--(-0.58,0.2);
\draw(p1)--(-0.58,-0.2);
\draw(p2)--(0.58,0.2);
\draw[gray](p2)--(0.58,-0.2);

\path(p1) edge [bend left=30] (p2);
\path[gray] (p1) edge [loop below, >=stealth,shorten >=1pt,looseness=20] (p1);
\draw[ultra thick, ->](1.5,0.1)--(2,0.1);
   \node [draw=white, fill=white] (a) at (1.7,-0.42)  {(b)};
\end{tikzpicture}
\hspace{0.1cm}   \begin{tikzpicture}[very thick,scale=0.8]
\tikzstyle{every node}=[circle, draw=black, fill=white, inner sep=0pt, minimum width=5pt];
\filldraw[fill=black!20!white, draw=black, thin, dashed](0,0)circle(0.8cm);
\node (po1) at (-0.3,0.3) {};
\node (pu1) at (-0.3,-0.3) {};
\node (p2) at (0.3,0) {};
\draw[gray](po1)--(p2);
\draw[gray](pu1)--(p2);
\draw(po1)--(pu1);
\draw[gray](po1)--(-0.58,0.4);
\draw(pu1)--(-0.58,-0.4);
\draw(p2)--(0.58,0.2);
\draw[gray](p2)--(0.58,-0.2);
\path (p2) edge [bend right=30] (po1);
\path[gray] (pu1) edge [loop below, >=stealth,shorten >=1pt,looseness=15] (pu1);
\end{tikzpicture}
\end{center}
\vspace{-0.2cm}
\caption{(a) Vertex-to-$K_4$-move. (b) Edge-to-$K_3$-move (vertex splitting). Gain labellings of edges are omitted.}
\label{fig:inductiveH3}
\end{figure}

\begin{definition}
An \emph{edge-to-$K_3$} move (also called \emph{vertex splitting} \cite{NOP,W2}) on a vertex $[v]$ which is incident to the edge $[v][u]$ with trivial gain and the edges $[v][u_i]$, $i=1,\ldots, t$ (which may include the edges $[v][u]$ and $[v][v]$ with gain $-1$), removes $[v]$ and its incident edges, and adds two new vertices $[v_0]$ and $[v_1]$ as well as the edges $[v_0][v_1]$, $[v_0][u]$ and $[v_1][u]$ with trivial gains.  Finally, each edge $[v][u_i]$ (with $[u_i]\neq [v]$), $i=1,\ldots, t$, is replaced by the edge $[v_0][u_i]$ or the edge $[v_1][u_i]$ so that the gain of the new edge $[v_j][u_i]$, $j\in \{0,1\}$, is the same as the gain of the deleted edge $[v][u_i]$. The loop at $[v]$ (if it exists) is replaced by a loop either at $[v_0]$ or $[v_1]$ with gain $-1$. 
\end{definition}

See Figure~\ref{fig:inductiveH3}(b).

For each of the above moves, an inverse move performed on a $(2,2,1)$-gain-tight signed quotient graph  is called \emph{admissible} if it results in another $(2,2,1)$-gain-tight signed quotient graph.

\begin{theorem} [Symmetrically isostatic graphs] \label{thm:forced2}
Let $\|\cdot\|_\P$ be a norm on $\bR^2$ for which $\P$ is a quadrilateral, and let $G$ be a $\mathbb{Z}_2$-symmetric graph where the action $\theta$ is free on the vertex set of $G$. Let $(G_0,\psi)$ be the signed quotient graph of $G$.
The following are equivalent.
\begin{enumerate}
\item[(i)] There exists a representation $\tau:\mathbb{Z}_2\to \Isom(\mathbb{R}^2)$, where $\tau(-1)$ is a reflection in the mirror $\ker\varphi_{F_1}$ along $\ker\varphi_{F_2}$, and a realisation $p$ such that the bar-joint framework $(G,p)$ is   well-positioned, $\mathbb{Z}_2$-symmetric and symmetrically isostatic in $(\bR^2,\|\cdot\|_\P)$;
\item[(ii)] $(G_0,\psi)$  is $(2,2,1)$-gain tight;
\item[(iii)] $(G_0,\psi)$ can be constructed from a  single unbalanced loop by a sequence of H1a,b,c-moves, H2a,b,c-moves, vertex-to-$K_4$ moves, and vertex splitting moves.
\end{enumerate}

\end{theorem}

\begin{proof} $(i) \Rightarrow (ii)$. Suppose $(G,p)$ is a well-positioned  symmetrically isostatic framework in $(\bR^2,\|\cdot\|_\P)$. Then we clearly have  $|E_0|= 2|V_0|-1$ since, by Lemma~\ref{lem:dimmotions}, the space of  symmetric infinitesimal trivial flexes is of dimension 1 (spanned by the infinitesimal translation along the mirror). Similarly, by Lemma~\ref{lem:symcounts}, there does not exist an  edge subset $F$ of $E_0$ with $|F|>2|V(F)|-1$, for otherwise the symmetric orbit  matrix of $(G_0,\psi)$ would have a row dependence.
So it remains to show that  we have  $|F|\leq 2|V(F)|-2$ for every balanced edge subset $F$. However, this also follows immediately from Lemma~\ref{lem:symcounts}. 

$(ii)\Rightarrow(iii)$. Suppose $(G_0,\psi)$  is $(2,2,1)$-gain tight. If $(G_0,\psi)$ is a single unbalanced loop, then we are done. So suppose $(G_0,\psi)$ has more than two vertices. Then $(G_0,\psi)$ has a vertex $[v]$ of degree $2$ or $3$. If there exists a vertex $[v]$  which is incident to two edges (one of which may be a loop), then there clearly exists an admissible inverse H1a,b- or c-move. If there is no such vertex, then there is a vertex $[v]$ which is incident to three non-loop edges, and $[v]$ has either two or three neighbours. If $[v]$ has two neighbours $[a]$ and $[b]$, and $[v],[a],[b]$ induce a graph with $5$ edges (i.e., a $2K_3-[e]$), then there exists an admissible inverse H2c-move. Otherwise, we may use the argument in \cite{anbs} for $(2,2,1)$-gain-tight signed graphs to show that there exists an admissible inverse H2b-move. If $[v]$ has three distinct neighbours, then it was again shown in \cite{anbs} that there exists an admissible inverse H2a-move for $[v]$, unless $[v]$ and its three neighbours $[a],[b]$ and $[c]$ induce a $K_4$ in $(G_0,\psi)$ with gain $1$ on every edge (plus possibly an additional edge with gain $-1$).

In this case there is an admissible inverse vertex-to-$K_4$ move, unless there exists a vertex $[x] \notin V(K_4)$ such that $[x][a]$ and $[x][b]$ are edges in $(G_0,\psi)$ which have the same gain.   
Let $A_0$ denote the $K_4$ and let $A_1$ be the graph consisting of $A_0$ together with the vertex $[x]$ and the edges $[x][a]$ and $[x][b]$.
By switching $[x]$, we may assume that the gains of $[x][a]$ and $[x][b]$ are both $1$. Note that  $[x][a]$ and $[x][b]$ cannot both have a parallel edge, and so, without loss of generality, we assume that the edge $[x][a]$ with gain $-1$ is not present. 

If there exists a vertex $[y]\notin V(A_1)$ and edges $[y][a]$ and $[y][x]$ with the same gain then let $A_2$ denote the union of $A_1$ with $[y]$ and these two edges (see Fig.~\ref{fig:pf}). By switching $[y]$ we may assume that all edges in $A_2$ have gain $1$.
Again, note that  $[y][a]$ and $[y][x]$ cannot both have a parallel edge, and so, without loss of generality, we assume that the edge $[y][a]$ with gain $-1$ is not present. If there exists a vertex $[z]\notin V(A_2)$ and edges $[z][y]$ and $[z][a]$ with the same gain then let $A_3$ denote the union of $A_2$ with $[z]$ and these two edges. Continuing this process we obtain an increasing sequence of subgraphs $A_1,A_2,A_3,\ldots$ of $G_0$ each of which is balanced and satisfies $|E(A_i)|=2|V(A_i)|-2$. This sequence must terminate after finitely many iterations at a subgraph $A_t$ of $G_0$.
Let $[w]$ be the vertex in $A_t\backslash A_{t-1}$ and suppose $[w]$ is incident to the vertices $[i]$ and $[j]$ in $A_{t-1}$.
By switching $[w]$ we may assume that all edges in $A_t$ have gain $1$.
By construction, one of the edges incident to $[w]$ in $A_t$, $[w][i]$ say, does not have a parallel edge and has the property that there is no vertex $[k]\notin V(A_t)$ which is adjacent to both $[w]$ and $[i]$ such that the edges $[k][w]$ and $[k][i]$ both have the same gain.

Clearly, there cannot exist a subgraph $H_0$ of $(G_0,\psi)$ with $|E(H_0)|=2|V(H_0)|-1$ which contains $[w]$ and $[i]$, but not $[j]$, for otherwise $A_{t}\cup H_0$ violates the $(2,2,1)$-gain-sparsity counts. To see this note that 
$|E(A_{t-1}\cup H_0)|=2|V(A_{t-1}\cup H_0)|-1$ and $A_{t}\cup H_0$
 is obtained by adjoining the edge $[w][j]$ to $A_{t-1}\cup H_0$.
Similarly,  there cannot exist a balanced subgraph $H_0$ of $(G_0,\psi)$ with $|E(H_0)|=2|V(H_0)|-2$ which contains $[w]$ and $[i]$, but not $[j]$.
To see this,  note that $A_t\cap H_0$ must be connected since otherwise $A_t\cup H_0$ violates the $(2,2,1)$-gain-sparsity counts.
By \cite[Lemma~2.5]{jkt}, $A_t\cup H_0$ is balanced and so, by Lemma \ref{switch}, we may assume every edge in $A_t\cup H_0$ has gain $1$. Note that $A_{t-1}$ and $H_0$ have a non-empty (balanced) intersection. 
Therefore, $|E(A_{t-1}\cup H_0)|=2|V(A_{t-1}\cup H_0)|-2$. However, if we add the edge $[w][j]$ to $A_{t-1}\cup H_0$, then this creates a balanced subgraph of $G_0$ which violates the $(2,2,1)$-gain-sparsity counts.
It follows that an inverse edge-to-$K_3$ move on the edge $[w][i]$ is admissible.

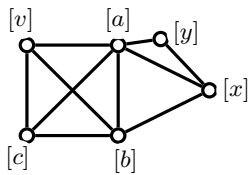
\begin{figure}[htp]
\begin{center}
\begin{tikzpicture}[very thick,scale=0.8]
\tikzstyle{every node}=[circle, draw=black, fill=white, inner sep=0pt, minimum width=5pt];
\node(p1) at (0,0) {};
\node(p2) at (1.5,0) {};
\node (p3) at (1.5,1.5) {};
\node(p4) at (0,1.5) {};
\node(p5) at (3,0.75) {};
\node(p6) at (2.2,1.6) {};

 \node [draw=white, fill=white] (a) at (-0.15,-0.38)  {\scriptsize$[c]$};
 \node [draw=white, fill=white] (a) at (-0.1,1.92)  {\scriptsize$[v]$};
 \node [draw=white, fill=white] (a) at (1.62,-0.42)  {\scriptsize$[b]$};
 \node [draw=white, fill=white] (a) at (1.5,1.92)  {\scriptsize$[a]$};

 \node [draw=white, fill=white] (a) at (3.44,0.75)  {\scriptsize$[x]$};
 \node [draw=white, fill=white] (a) at (2.62,1.7)  {\scriptsize$[y]$};

\draw(p1)--(p2);
\draw(p2)--(p3);
\draw(p3)--(p4);
\draw(p1)--(p4);
\draw(p3)--(p1);
\draw(p2)--(p4);
\draw(p2)--(p5);
\draw(p3)--(p5);

\draw(p6)--(p5);
\draw(p6)--(p3);

\end{tikzpicture}
\end{center}
\vspace{-0.2cm}
\caption{Illustration of the subgraph $A_2$ in the proof of Theorem~\ref{thm:forced2} (ii) $\Rightarrow$ (iii). All edges have gain $1$. }
\label{fig:pf}
\end{figure}

$(iii)\Rightarrow (i)$. We employ induction on the number of vertices of $G_0$. If $G_0$ is a single unbalanced loop with vertex $[v]$, choose $p_v\notin \ker \varphi_{F_1}$ and set $p_{-v}=\tau(-1)p_v$. Then $(G,p)$ is well-positioned and $\bZ_2$-symmetric and so the statement holds  by Lemma~\ref{K2Lemma}.

Now, let $n\geq 2$, and suppose $(i)$ holds for all signed quotient   graphs satisfying $(iii)$ with at most $n-1$ vertices. Let $(G_0,\psi)$ have $n$ vertices, and let $(G'_0,\psi')$ be the penultimate graph in the construction sequence of $(G_0,\psi)$. If $(G'_0,\psi')$ is a single unbalanced loop, then $(G_0,\psi)$ is obtained from $(G'_0,\psi')$ by a H1b-,  H1c-, or  vertex-to-$K_4$ move. The loop of $G'_0$ belongs to the induced monochrome subgraph $G'_{F_1,0}$ of $G'_0$, and for each of the three moves, it is easy to see how to place the new vertex (vertices) so that the  induced monochrome subgraphs $G_{F_1,0}$ and $G_{F_2,0}$ of $G_0$ have the property that $G_{F_1,0}$ is a spanning unbalanced map graph  and $G_{F_2,0}$ is a spanning tree of $G_0$ (see also the discussion below). The result then follows from Theorem~\ref{thm:forced1}.  Thus, we may assume that $G'_0$ has at least two vertices.

In this case, it follows from the induction hypothesis and Theorem~\ref{thm:forced1} that there exists a well-positioned $\mathbb{Z}_2$-symmetric realisation $p'$ of the covering graph $G'$ of $(G'_0,\psi')$  in $(\bR^2,\|\cdot\|_\P)$ (where the reflection $\tau(-1)$ is in the mirror $\ker\varphi_{F_1}$)  so that the induced monochrome subgraphs $G'_{F_1,0}$ and $G'_{F_2,0}$ of $G'_0$ are both spanning, 
$G'_{F_1,0}$ is an unbalanced map graph,  and $G'_{F_2,0}$ is a tree. By Theorem~\ref{thm:forced1} it now suffices to show that the vertex (or vertices) of $G\setminus G'$ can be placed in such a way that the corresponding framework $(G,p)$ is $\mathbb{Z}_2$-symmetric and well-positioned, the induced monochrome subgraphs $G_{F_1,0}$ and $G_{F_2,0}$ are both spanning in $G_0$, $G_{F_1,0}$ is an unbalanced map graph, and $G_{F_2,0}$ is a tree.

Choose points $x_1$ and $x_2$ in the relative interiors of $F_1$ and $F_2$ respectively.
Suppose first that $(G_0,\psi)$ is obtained from $(G'_0,\psi')$ by a H1a-move, where $[v]\in G_0\setminus G'_0$ is adjacent to the vertices $[v_1]$ and $[v_2]$ of $G'_0$ with respective gains $\gamma_1$ and $\gamma_2$. 
Set $p_w=p_w'$ for all vertices $w$ in $G$ with $[w]\not=[v]$.
Let $a\in \bR^2$ be the point of intersection of the lines $L_1=\{\tau(\gamma_1)p_{\tv_1}+tx_1:t\in \bR\}$ and $L_2=\{\tau(\gamma_2)p_{\tv_2}+tx_2:t\in \bR\}$ and let $B(a,r)$ be an open ball with centre $a$ and radius $r>0$. Choose $p_{\tv}$ to be any point in $B(a,r)$ which is distinct from $\{p_w:w\in V(G')\}$ and which is not fixed by $\tau(-1)$. Set $p_{-\tv}=\tau(-1)p_{\tv}$.
Then $(G,p)$ is a $\bZ_2$-symmetric bar-joint framework and, by  applying a small perturbation to $p_{\tilde{v}}$ if necessary, we may assume that $(G,p)$ is well-positioned. 
If $r$ is sufficiently small then the induced framework colours for $[v][v_1]$ and $[v][v_2]$ are $[F_1]$ and $[F_2]$ respectively. 
Thus, the induced monochrome subgraphs of $(G_0,\psi)$ are $G_{F_1,0}=G'_{F_1,0}\cup \{[v][v_1]\}$ and $G_{F_2,0}=G'_{F_2,0}\cup \{[v][v_2]\}$.  Clearly, $G_{F_1,0}$  is a spanning unbalanced map graph and $G_{F_2,0}$ is a spanning tree of $G_0$. For an illustration of the monochrome subgraphs of the signed quotient graph see Fig.~\ref{fig:inductive}(a). The edges of $G_{F_1,0}$ are shown in gray and the edges of $G_{F_2,0}$ are shown in black.

If $(G_0,\psi)$ is obtained from $(G'_0,\psi)$ by a H1b-move, then the proof is completely analogous to the proof above. (See Fig.~\ref{fig:inductive}(b)).

Suppose $(G_0,\psi)$ is obtained from $(G'_0,\psi')$ by a H1c-move, where $[v]\in G_0\setminus G'_0$ is incident to the unbalanced loop $[e]$ and  adjacent to the vertex $[w]$ of $(G'_0,\psi')$ with gain $\gamma$.  If we choose $p_{\tv}$ to be any point on the line $L=\{\tau(\gamma)p_{\tw}+tx_2:t\in \bR\}$, then the induced framework colouring for $[v][w]$ is $[F_2]$. Moreover, as we have seen before, the induced framework colouring for the loop $[e]$ is $[F_1]$. It follows that we may place $\tv$ and $-\tv$ in such a way that  $(G,p)$ is well-positioned and $\bZ_2$-symmetric, and the induced monochrome subgraphs of $G_0$ are $G_{F_1,0}=G'_{F_1,0}\cup \{[e]\}$ and $G_{F_2,0}=G'_{F_2,0}\cup \{[v][w]\}$.  Clearly, $G_{F_1,0}$ is  an unbalanced spanning map graph and $G_{F_2,0}$ is a spanning tree of $(G_0,\psi)$.
(See Fig.~\ref{fig:inductive}(c)).

Next, we suppose that $(G_0,\psi)$ is obtained from $(G'_0,\psi')$ by a H2a-move where $[v]\in G_0\setminus G'_0$ subdivides the edge $[e]$ into the edges $[e_1]$ and $[e_2]$ with respective gains $\gamma_1$ and $\gamma_2$, and $[v]$ is also incident to the edge $[e_3]$ with end-vertex $[z]$ and gain $\gamma_3$. Without loss of generality we may assume that $[e]\in G'_{F_1,0}$. Let $a\in \bR^2$ be the point of intersection of the line $L_1$ which passes through the points $\tau(\gamma_1)p_{\tv_1}$ and $\tau(\gamma_2)p_{\tv_2}$ with $L_2=\{\tau(\gamma_3)p_{\tz}+tx_2:t\in\bR\}$.
Let $B(a,r)$ be the open ball with centre $a$ and radius $r>0$ and choose $p_{\tv}$ to be a point in $B(a,r)$ which is distinct from $\{p_w:w\in G'\}$ and which is not fixed by $\tau(-1)$.
Set $p_{-\tv}=\tau(-1)p_{\tv}$. As above, $(G,p)$ is $\bZ_2$-symmetric and we may assume it is well-positioned.
If $r$ is sufficiently small then  $[e_1]$ and $[e_2]$ have induced framework colour $[F_1]$ and $[e_3]$ has framework colour $[F_2]$.
The induced monochrome subgraphs of $G_0$ are $G_{F_1,0}=(G'_{F_1,0}\backslash\{[e]\})\cup \{[e_1],[e_2]\}$ and $G_{F_2,0}=G'_{F_2,0}\cup \{[e_3]\}$. Clearly, $G_{F_1,0}$ is a spanning unbalanced map graph and $G_{F_2,0}$ is a spanning tree of $G_0$. (See Fig.~\ref{fig:inductiveH2}(a)).

The cases where $(G_0,\psi)$ is obtained from $(G'_0,\psi')$ by a H2b- or a H2c-move can be proved completely  analogously to the case above for the H2a-move. Note, however, that for the H2c-move, the edges $[e_1]$ and $[e_2]$ are forced to be in the subgraph $G_{F_1,0}$. (See Fig.~\ref{fig:inductiveH2}(b),(c)).

Next, we suppose that $(G_0,\psi)$ is obtained from $(G'_0,\psi')$ by a vertex-to-$K_4$-move, where the vertex $[v]$ of  $G'_0$ (which may be incident to an unbalanced loop $[e]$) is replaced by a copy of $K_4$ with a trivial gain labelling (and $[e]$ is replaced by the edge $[f]$).  It was shown in \cite[Ex.~4.5]{kit-pow} that $K_4$ has a well-positioned and isostatic placement in $(\bR^2,\|\cdot\|_\P)$. Moreover, we may scale this realisation  so that all of the vertices of the $K_4$ lie in a ball of arbitrarily small radius. For any such  realisation, the induced monochrome subgraphs of $K_4$ are both paths of length $3$. 
Let $B(p_{\tv},r)$ be the open  ball with centre $p_{\tv}$ and radius $r>0$. Choose a placement of the representative vertices of the new $K_4$ to lie within $B(p_{\tv},r)$ such that the vertices are distinct from $\{p_w:w\in V(G')\backslash\{\tv,-\tv\}\}$, none of the vertex placements are fixed by $\tau(-1)$ and the resulting placement of the new $K_4$ is isostatic. If $r$ is sufficiently small then the edge  $[f]$ (if present) has the induced framework colour $[F_1]$.
It can be assumed that the corresponding $\bZ_2$-symmetric placement of $G$ is well-positioned. Moreover, the induced monochrome subgraphs $G_{F_1,0}$ and $G_{F_2,0}$ of $G_0$ clearly have the desired properties. (See Fig.~\ref{fig:inductiveH3}(a)).

Finally, we suppose that $(G_0,\psi)$ is obtained from $(G'_0,\psi')$ by an edge-to-$K_3$-move, where the vertex $[v]$ of  $G'_0$ (which  is replaced by the vertices $[v_0]$ and $[v_1]$) is incident to the edge $[v][u]$ with trivial gain and the edges $[v][u_i]$, $i=1,\ldots, t$, in $G'_0$. Without loss of generality we may assume that $[v][u]\in G'_{F_1,0}$. If we choose $p_{\tv_0}=p_{\tv}$ and $p_{\tv_1}$ to be a point on the line $L=\{p_{\tv}+tx_2:t\in \bR\}$ which is sufficiently close to $p_{\tv}$, then  the induced framework colour for $[v_0][v_1]$ is $[F_2]$ and the induced framework colour for $[v_0][u]$ and $[v_1][u]$ is $[F_1]$. (Again we may assume the framework is well-positioned). Moreover, all other edges of $G'_0$ which have been replaced by new edges in $G_0$ clearly retain their induced framework colouring if $p_{\tv_1}$   is chosen sufficiently close to $p_{\tv}$.   It is now easy to see that for such a  placement of $\tv_0$ and $\tv_1$, $(G,p)$ is $\bZ_2$-symmetric and for the induced monochrome subgraphs  $G_{F_1,0}$ and $G_{F_2,0}$   of $G_0$ we have that $G_{F_1,0}$ is  a spanning unbalanced  map graph and $G_{F_2,0}$ is a spanning tree of $(G_0,\psi)$. (See Fig.~\ref{fig:inductiveH3}(b)).
This completes the proof.    

\end{proof}

\begin{figure}[htp]
\begin{center}
\begin{tikzpicture}[very thick,scale=0.8]
\tikzstyle{every node}=[circle, draw=black, fill=white, inner sep=0pt, minimum width=5pt];

\node (p1) at (0,0) {};
\node (p2) at (4,0) {};
\node (p1o) at (1,0.5) {};
\node (p2o) at (3,0.5) {};
\node (p1l) at (0.5,1.5) {};
\node (p2r) at (3.5,1.5) {};
\draw[gray](p1)--(p1o);
\draw[gray](p2)--(p2o);
\draw[gray](p1)--(p2o);
\draw[gray](p2)--(p1o);
\draw(p1l)--(p1o);
\draw[gray](p1l)--(p2o);
\draw(p1l)--(p1);
\draw(p2r)--(p2o);
\draw[gray](p2r)--(p1o);
\draw(p2r)--(p2);
\draw[dashed,thin](2,-0.3)--(2,2);
   \node [draw=white, fill=white] (a) at (2,-0.9)  {(a)};

 \node [draw=white, fill=white] (a) at (4.8,0.5)  {\tiny $1$};
 \node [draw=white, fill=white] (a) at (5.3,0.6)  {\tiny $1$};

 \node [draw=white, fill=white] (a) at (5.29,0.3)  {\tiny $1$};
 \node [draw=white, fill=white] (a) at (5.6,1.2)  {\tiny $-1$};

 \node [draw=white, fill=white] (a) at (5.6,-0.2)  {\tiny $-1$};

\node(p11) at (5,0) {};
\node(p22) at (5,1) {};
\node (p33) at (6,0.5) {};
\draw(p22)--(p11);
\draw(p22)--(p33);
\draw[gray](p33)--(p11);
\path(p11) edge [gray,bend right=40] (p33);
\path(p33) edge [gray,bend right=40] (p22);
\node [draw=white, fill=white] (a) at (5.5,-0.9)  {(b)};

\node (p1) at (9.5,0) {};
\node (p2) at (10.5,0) {};
\node (p1o) at (9.5,1.5) {};
\node (p2o) at (10.5,1.5) {};
\node (p1l) at (8.5,0) {};
\node (p2r) at (11.5,0) {};
\draw(p1)--(p1o);
\draw(p2)--(p2o);
\draw(p1)--(p2o);
\draw(p2)--(p1o);
\draw(p1l)--(p1o);
\draw[gray](p1l)--(p2o);
\draw[gray](p1l)--(p1);
\draw(p2r)--(p2o);
\draw[gray](p2r)--(p1o);
\draw[gray](p2r)--(p2);
\draw[dashed,thin](10,-0.3)--(10,2);
   \node [draw=white, fill=white] (a) at (10,-0.9)  {(c)};

 \node [draw=white, fill=white] (a) at (14.6,0.4)  {\tiny $-1$};
 \node [draw=white, fill=white] (a) at (13.8,0.4)  {\tiny $1$};

 \node [draw=white, fill=white] (a) at (12.9,0.7)  {\tiny $-1$};
 \node [draw=white, fill=white] (a) at (13.4,0.59)  {\tiny $1$};

 \node [draw=white, fill=white] (a) at (13.5,-0.2)  {\tiny $1$};

\node (p11) at (13,0) {};
\node (p22) at (14,0) {};
\node (p33) at (14,1) {};
\draw[gray](p11)--(p22);
\draw(p22)--(p33);
\draw(p11)--(p33);
\path(p22) edge [bend right=40] (p33);
\path[gray](p11) edge [bend left=40] (p33);
 \node [draw=white, fill=white] (a) at (13.5,-0.9)  {(d)};

\end{tikzpicture}
\end{center}
\vspace{-0.2cm}
\caption{A  symmetrically isostatic (but not anti-symmetrically isostatic) reflection framework in $(\mathbb{R}^2,\|\cdot\|_\infty)$ (a) and its signed quotient graph $(G_0,\psi)$ (b).  An anti-symmetrically isostatic (but not symmetrically isostatic) reflection framework in $(\mathbb{R}^2,\|\cdot\|_\infty)$ (c) with the same signed quotient graph $(G_0,\psi)$. The edges of the induced monochrome subgraphs $G_{F_1}$ and  $G_{F_1,0}$ are shown in gray colour. $(G_0,\psi)$ does not admit an infinitesimally rigid realisation in $(\mathbb{R}^2,\|\cdot\|_\infty)$ with reflection symmetry since $|E_0|<2|V_0|$.}
\label{fig:antisym}
\end{figure}
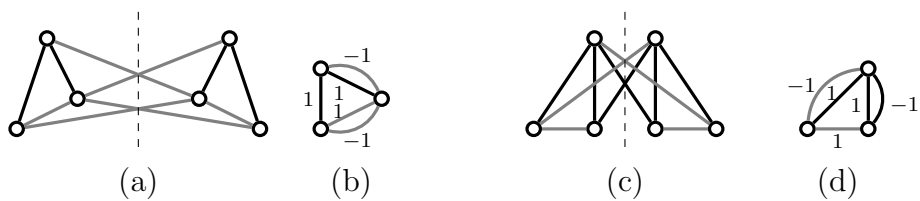

\begin{example} \label{ex:2k3minusedge} The smallest signed quotient graph $(G_0,\psi)$ whose covering graph $G$ can be realised as a $\mathbb{Z}_2$-symmetric framework in $(\mathbb{R}^2,\|\cdot\|_\P)$ which is anti-symmetrically isostatic is the graph $2K_3-\te$ shown in Figure~\ref{fig:antisym} (b,d).   Figure~\ref{fig:antisym} (c) illustrates such a realisation $(G,p)$ in $(\mathbb{R}^2,\|\cdot\|_\infty)$.
To obtain a realisation $(G,\tilde{p})$ in $(\mathbb{R}^2,\|\cdot\|_\P)$ construct a linear isometry $T:(\mathbb{R}^2,\|\cdot\|_\infty) \to 
(\mathbb{R}^2,\|\cdot\|_\P)$ and set $\tilde{p}_v=T(p_v)$ for each $v\in V$. 
\end{example}

A \emph{$2K_3-[e]$ edge joining} move joins a signed quotient graph $2K_3-[e]$  to $(G_0,\psi)$ via one new edge of arbitrary gain, where $2K_3-[e]$ consists of $3$ vertices and $5$ edges.

\begin{theorem} [Anti-symmetrically isostatic graphs] \label{thm:anti2}
Let $\|\cdot\|_\P$ be a norm on $\bR^2$ for which $\P$ is a quadrilateral, and let $G$ be a $\mathbb{Z}_2$-symmetric graph with respect to the action $\theta$ which is free on the vertex set of $G$. Let $(G_0,\psi)$ be the signed quotient  graph of $G$.
The following are equivalent.
\begin{enumerate}
\item[(i)] There exists a representation $\tau:\mathbb{Z}_2\to \Isom(\mathbb{R}^2)$, where $\tau(-1)$ is a reflection in the mirror $\ker\varphi_{F_1}$ along $\ker\varphi_{F_2}$, and a realisation $p$ such that the bar-joint framework $(G,p)$ is   well-positioned, $\mathbb{Z}_2$-symmetric and anti-symmetrically isostatic in $(\bR^2,\|\cdot\|_\P)$;

\item[(ii)] $(G_0,\psi)$ has no loops and is $(2,2,1)$-gain tight;

\item[(iii)] $(G_0,\psi)$ can be constructed from $2K_3-[e]$ by a sequence of H1a,b-moves, H2a,b-moves, vertex-to-$K_4$ moves, vertex splitting moves and $2K_3-[e]$ edge joining moves.

\end{enumerate}

\end{theorem}

\begin{proof}  $(i) \Rightarrow (ii)$. Suppose $(G,p)$ is a well-positioned anti-symmetrically isostatic framework in $(\bR^2,\|\cdot\|_\P)$.  Then, by Lemma~\ref{AntiSymFixedEdges}, $(G_0,\psi)$ cannot contain a loop.  
 The rest of the proof is completely analogous to the proof of Theorem~\ref{thm:forced2} ((i) $\Rightarrow$ (ii)), since the space of anti-symmetric infinitesimal trivial flexes is also of dimension 1, by Lemma~\ref{lem:dimmotions}.  

$(ii) \Rightarrow (iii)$. Suppose $(G_0,\psi)$  is $(2,2,1)$-gain tight with no loops. If $(G_0,\psi)$ is a $2K_3-[e]$, then we are done. So suppose $(G_0,\psi)$ has more than three vertices. Then $(G_0,\psi)$ has a vertex $[v]$ of degree $2$ or $3$. It was shown in \cite{anbs} that there exists an admissible inverse Henneberg 1a,b- or 2a,b-move for $[v]$, unless $[v]$ either  has three distinct neighbours $[a], [b]$ and $[c]$ in $(G_0,\psi)$ and  $[v], [a], [b], [c]$ induce a $K_4$ with gain $1$ on every edge (plus possibly an additional edge with gain $-1$) or  $[v]$ has two distinct neighbors $[a]$ and $[b]$,  and  $[v], [a], [b]$ induce a  $2K_3-[e]$.

In the first case, there is an admissible inverse vertex-to-$K_4$ move or an admissible inverse vertex splitting move, as shown in the proof of Theorem~\ref{thm:forced2} ($(ii) \Rightarrow (iii)$). Thus, we may assume that every vertex of degree $3$ is in a copy of $2K_3-[e]$. But now we may use a similar argument as in the proof for the characterisation of $(2,2,1)$-gain-tight signed quotient graphs given  in \cite{anbs} (see also \cite[Lemma 4.10]{nixowen}) to show that at least one of the copies of $2K_3-[e]$ has the property that there is exactly one edge which joins a vertex $[x]\notin 2K_3-[e]$ with a vertex in $2K_3-[e]$. For a signed quotient graph $(H,\phi)$ with vertex set $V(H)$ and edge set $E(H)$, we define $f(H)=2|V(H)|-|E(H)|$. Let $Y=\{Y_1,\ldots, Y_k\}$ be the copies of $2K_3-[e]$ in $(G_0,\psi)$. Then the $Y_i$ are pairwise disjoint and satisfy $f(Y_i)=1$ for all $i$. Let $W_0$ and $F_0$ be the sets of vertices and edges of $(G_0,\psi)$ which do not belong to any of the $Y_i$. Then we have $f(G_0)=\sum_{i=1}^k f(Y_i)+2|W_0|-|F_0|$, and since $f(G_0)=1$, $|F_0|=2|V_0|+k-1$. Every vertex in $W_0$ is of degree at least $4$. So if every $Y_i$ is incident to at least two edges in $F_0$, then there are at least $4|W_0|+2k$ edge-vertex incidences for the edges in $F_0$. But then we have $|F_0|\geq 2|W_0|+k$, a contradiction. If there exists a $Y_i$ with the property that none of the vertices of $Y_i$ are incident with an edge in $F_0$, then $G_0=Y_i$, contradicting our assumption that $G_0$ has more than $3$ vertices. It follows that there exists an inverse $2K_3-[e]$ edge joining move.

$(iii) \Rightarrow (i)$. We employ induction on the number of vertices. For the signed graph $2K_3-[e]$, the statement follows from Example~\ref{ex:2k3minusedge}.

Now, let $n\geq 4$, and suppose (i) holds for all signed quotient  graphs satisfying (iii) with at most $n-1$ vertices. Let $(G_0,\psi)$ have $n$ vertices, and  suppose first that the last move in the construction sequence of $(G_0,\psi)$ is not a $2K_3-[e]$ edge joining move.   Then we let $(G'_0,\psi')$ be the penultimate graph in the construction sequence of $(G_0,\psi)$. By the induction hypothesis and Theorem~\ref{thm:antisym1}, there exists a well-positioned $\mathbb{Z}_2$-symmetric realisation of the covering graph of $(G'_0,\psi')$  in $(\bR^2,\|\cdot\|_\P)$ (where the reflection $\tau(-1)$ is in the mirror $\ker\varphi_{F_1}$)  so that the induced monochrome subgraphs $G'_{F_1,0}$ and $G'_{F_2,0}$ of $G'_0$ are both spanning, 
$G'_{F_1,0}$ is a  tree,  and $G'_{F_2,0}$ is an unbalanced map graph. By Theorem~\ref{thm:antisym1} it suffices to show that the vertex (or vertices) of $G\setminus G'$ can be placed so that $(G,p)$ is well-positioned, $\bZ_2$-symmetric and the induced monochrome subgraphs $G_{F_1,0}$ and $G_{F_2,0}$ of $G_0$ are both spanning, $G_{F_1,0}$ is a tree  and $G_{F_2,0}$ is an unbalanced map graph.

If $(G_0,\psi)$ is obtained from $(G'_0,\psi')$ by a H1a-, H1b-, H2a-, H2b-, vertex-to-$K_4$, or edge-to-$K_3$ move, then we may use exactly the same placement for the vertex (or vertices) of  $G\setminus G'$ as in the proof of Theorem~\ref{thm:forced2} to obtain the desired realisation of $G$. 

So it remains to consider the case where the last move in the construction sequence of $(G_0,\psi)$ is a $2K_3-[e]$ edge joining move. Suppose $(G_0,\psi)$  is obtained by joining the  signed quotient graphs $(G'_0,\psi')$ and $(G''_0,\psi'')$  by an edge $[f]$ with end-vertices $[u]\in G'_0$ and $[v]\in G''_0$, where $G''_0=2K_3-[e]$. By the induction hypothesis, Theorem~\ref{thm:antisym1}, and Example~\ref{ex:2k3minusedge}, the covering graphs of  $(G'_0,\psi')$  and   $(G''_0,\psi'')$ can be realised as $\mathbb{Z}_2$-symmetric frameworks $(G',p)$ and $(G'',q)$  in $(\bR^2,\|\cdot\|_\P)$ (where the reflection $\tau(-1)$ is in the mirror $\ker\varphi_{F_1}$)  so that the induced monochrome subgraphs $G'_{F_1,0}$ and $G'_{F_2,0}$ of $G'_0$, and   $G''_{F_1,0}$ and $G''_{F_2,0}$ of $G''_0$, are all spanning, 
$G'_{F_1,0}$ and $G''_{F_1,0}$  are  trees,  and $G'_{F_2,0}$ and $G''_{F_2,0}$  are unbalanced map graphs. 
Now, consider the line $L$ which passes through the points $p_{\tu}$ and $\tau(-1)p_{\tu}$, and translate the framework  $(G'',q)$ along the mirror line $\ker\varphi_{F_1}$ (thereby preserving the reflection symmetry of $(G'',q)$) so that the points $\hat q_{\tv}$ and $\tau(-1)\hat q_{\tv}$ of the translated framework $(G'',\hat q)$ lie on $L$. If there are vertices of $(G',p)$ and $(G'',\hat q)$ which are now positioned at the same point in $(\bR^2,\|\cdot\|_\P)$, then we perturb the vertices of $(G'',\hat q)$ slightly without changing the induced colourings of the edges of $G$ until all of the vertices  have different positions.   Then  $[f]$ has induced framework colour $[F_1]$, the realisation  of  $G$ is well-positioned, and the induced monochrome subgraphs of $G_0$ are $G_{F_1,0}=G'_{F_1,0}\cup G''_{F_1,0}\cup\{[f]\}$ and $G_{F_2,0}=G'_{F_2,0}\cup G''_{F_2,0}$. Clearly, $G_{F_1,0}$ is a spanning tree  and $G_{F_2,0}$ is a spanning unbalanced map graph of $G_0$.

\end{proof}

Note that the final argument in the proof of Theorem~\ref{thm:anti2} can immediately be generalised to show that in the recursive  construction sequence in Theorem~\ref{thm:anti2} $(iii)$, we may replace the $2K_3-[e]$ edge joining move with an edge joining move  that joins two \emph{arbitrary} $(2,2,1)$-gain tight signed quotient graphs by an edge of arbitrary gain.

\section{Further remarks}

At the graph level, we provided characterisations for  symmetric and anti-symmetric infinitesimal rigidity in terms of gain-sparsity counts and recursive constructions (see Theorems~\ref{thm:forced2} and \ref{thm:anti2}). However, a characterisation in terms of monochrome subgraph decompositions (analogous to the results in Section~\ref{Sect:framework}) was not given, as it is not clear whether for an arbitrary decomposition of a signed quotient graph into a monochrome spanning unbalanced map graph and a monochrome spanning tree, there always exists a grid-like realisation of the covering graph with reflectional symmetry which respects the given edge colourings. These realisation problems are non-trivial \cite{kit-sch,kit-sch1} and even arise in the non-symmetric situation \cite{kitson}. 

It is easy to see that a necessary count for the existence of a $2$-dimensional infinitesimally rigid grid-like $\mathbb{Z}_2$-symmetric  realisation of a graph $G$ is that its  signed quotient graph   $(G_0,\psi)$ contains a spanning subgraph with $F$ edges which is \emph{$(2,2,0)$-gain-tight}, i.e., $|F|= 2|V(F)|$, $|F'|\leq 2|V(F')|-2$ for every balanced $F'\subseteq F$, and $|F'|\leq 2|V(F')|$ for every  $F'\subseteq F$. This is because $(G_0,\psi)$ needs to contain two monochrome   connected unbalanced spanning map graphs, by Corollary \ref{cor}. However, these conditions are clearly not sufficient.

Finally, it is natural to ask whether the results of this paper can be extended to grid-like frameworks in the plane with half-turn symmetry. A necessary condition for a grid-like half-turn-symmetric framework to be  symmetrically isostatic is that the associated signed quotient graph $(G_0,\psi)$ satisfies $|E_0|= 2|V_0|$, as  there are no  symmetric trivial infinitesimal flexes with respect to the half-turn symmetry group. In fact, $(G_0,\psi)$ must clearly be $(2,2,0)$-gain-tight. A combinatorial characterisation of  $(2,2,0)$-gain-tight graphs, however, has not yet been obtained (see also \cite{anbs}). For anti-symmetric isostaticity, the situation is much easier, as we need  $(G_0,\psi)$ to satisfy $|E_0|= 2|V_0|-2$ and $|F|\leq 2|V(F)|-2$ for every  $F\subseteq E_0$, and these types of signed quotient graphs have been described in \cite{anbs}. 

More generally, it would of course also be of interest to extend the results of this paper to frameworks with larger symmetry groups and to different normed spaces.

\end{document}